\documentclass{amsart}

\usepackage{amsthm}
\usepackage{amsmath}
\usepackage{amssymb}
\usepackage{amsfonts}
\usepackage{latexsym}
\usepackage[all]{xy} \SelectTips{eu}{}
\usepackage{hyperref}
\usepackage{eucal}
\usepackage{xcolor}
\usepackage{tikz}
\usepackage{wasysym}
\usepackage{mathrsfs}

\newtheorem{introthm}{Theorem}[]

\newtheorem*{intque*}{Question}
\newtheorem*{intexa*}{Example}
\newcommand{\numberseries}{\bfseries}   
\newlength{\thmtopspace}                
\newlength{\thmbotspace}                
\newlength{\thmheadspace}               
\newlength{\thmindent}                  
\newtheoremstyle{bfupright head,slanted body}
{\thmtopspace}{\thmbotspace}
{\slshape}{\thmindent}{\bfseries}{.}{\thmheadspace}
{{\numberseries \thmnumber{#2\;}}\thmnote{#3}}

\newtheoremstyle{bfupright head,upright body}
{\thmtopspace}{\thmbotspace}
{\upshape}{\thmindent}{\bfseries}{.}{\thmheadspace}
{{\numberseries \thmnumber{#2\;}}\thmnote{#3}}

\newtheoremstyle{fixed bf head,slanted body}
{\thmtopspace}{\thmbotspace}{\slshape}
{\thmindent}{\bfseries}{.}{\thmheadspace}
{{\numberseries \thmnumber{#2\;}}\thmname{#1}\thmnote{ (#3)}}

\newtheoremstyle{fixed bf head,upright body}
{\thmtopspace}{\thmbotspace}{\upshape}
{\thmindent}{\bfseries}{.}{\thmheadspace}
{{\numberseries \thmnumber{#2\;}}\thmname{#1}\thmnote{ (#3)}}

\newtheoremstyle{numbered paragraph}
{\thmtopspace}{\thmbotspace}{\upshape}
{\thmindent}{\upshape}{}{\thmheadspace}
{{\numberseries \thmnumber{#2.}}}

\theoremstyle{bfupright head,slanted body}
\newtheorem{res}{}[section]             \newtheorem*{res*}{}

\theoremstyle{bfupright head,upright body}
\newtheorem{bfhpg}[res]{}               \newtheorem*{bfhpg*}{}

\theoremstyle{fixed bf head,slanted body}
\newtheorem{thm}[res]{Theorem}          \newtheorem*{thm*}{Theorem}
\newtheorem{prp}[res]{Proposition}      \newtheorem*{prp*}{Proposition}
\newtheorem{cor}[res]{Corollary}        \newtheorem*{cor*}{Corollary}
\newtheorem{lem}[res]{Lemma}            \newtheorem*{lem*}{Lemma}
         \newtheorem*{que*}{Question}
\theoremstyle{fixed bf head,upright body}
\newtheorem{dfn}[res]{Definition}       \newtheorem*{dfn*}{Definition}
\newtheorem{rmk}[res]{Remark}           \newtheorem*{rmk*}{Remark}
\newtheorem{exa}[res]{Example}           \newtheorem*{exa*}{Example}
\newtheorem{setup}[res]{Setup}           \newtheorem*{setup*}{Setup}
           \newtheorem*{nota*}{Notation}

\theoremstyle{numbered paragraph}
\newtheorem{ipg}[res]{}

\newlength{\thmlistleft}        
\newlength{\thmlistright}       
\newlength{\thmlistpartopsep}   
\newlength{\thmlisttopsep}      
\newlength{\thmlistparsep}      
\newlength{\thmlistitemsep}     

\newcounter{eqc}
\newenvironment{eqc}{\begin{list}{\upshape (\textit{\roman{eqc}})}%
		{\usecounter{eqc}%
			\setlength{\leftmargin}{\thmlistleft}%
			\setlength{\labelwidth}{\thmlistleft}%
			\setlength{\rightmargin}{\thmlistright}%
			\setlength{\partopsep}{\thmlistpartopsep}%
			\setlength{\topsep}{\thmlisttopsep}%
			\setlength{\parsep}{\thmlistparsep}%
			\setlength{\itemsep}{\thmlistitemsep}}}%
	{\end{list}}%

\newcounter{prt}
\newenvironment{prt}{\begin{list}{\upshape (\alph{prt})}%
		{\usecounter{prt}%
			\setlength{\leftmargin}{\thmlistleft}%
			\setlength{\labelwidth}{\thmlistleft}%
			\setlength{\rightmargin}{\thmlistright}%
			\setlength{\partopsep}{\thmlistpartopsep}%
			\setlength{\topsep}{\thmlisttopsep}%
			\setlength{\parsep}{\thmlistparsep}%
			\setlength{\itemsep}{\thmlistitemsep}}}%
	{\end{list}}%

\newcounter{rqm}
	{\end{list}}%

	{\end{list}}%

\newenvironment{prf*}[1][Proof]{%
	\begin{proof}[\bf #1]
		\setcounter{equation}{0}
		}
	{\end{proof}
}

\numberwithin{equation}{res}
\allowdisplaybreaks[1]
\hypersetup{pdftitle={A Krause-type recollement for Q-shaped derived categories},
pdfauthor={Yajun Ma and Peiru Yang},pdfsubject={Q-shaped homological algebra},
pdfkeywords={cotorsion pairs, recollement, compact generation, objectwise injectives}}

\newcommand{\kk}{\Bbbk}
\newcommand{\Q}{Q}
\newcommand{\I}{\mathscr I}
\newcommand{\EE}{\mathscr E}
\newcommand{\EI}{\mathscr E_{\mathrm{inj}}}
\newcommand{\WI}{\mathscr W_{\mathrm{inj}}}
\newcommand{\UU}{\mathscr U}
\newcommand{\FF}{\mathscr F_{\mathrm{inj}}}
\newcommand{\rQ}{\mathfrak r}
\newcommand{\KQ}{\mathcal K_{\Q}(A)}
\newcommand{\KI}{\mathcal K_{\Q}({}_A\Inj)}
\newcommand{\KAI}{\mathcal K_{\Q}^{\mathrm{ac}}({}_A\Inj)}
\newcommand{\KQA}{\mathcal K_{\Q}^{\mathrm{ac}}(A)}
\newcommand{\DQ}{\mathcal D_{\Q}(A)}
\newcommand{\X}{\mathcal{X}}
\newcommand{\DD}{\mathcal{D}}
\newcommand{\Y}{\mathcal{Y}}
\newcommand{\Mod}[1]{{}_{#1}\operatorname{Mod}}
\DeclareMathOperator{\modu}{mod}
\DeclareMathOperator{\Inj}{Inj}
\DeclareMathOperator{\Prj}{Prj}
\DeclareMathOperator{\prj}{prj}
\DeclareMathOperator{\Hom}{Hom}
\DeclareMathOperator{\Ext}{Ext}
\DeclareMathOperator{\Ker}{Ker}
\DeclareMathOperator{\Coker}{Cok}
\renewcommand{\Im}{\operatorname{Im}}
\DeclareMathOperator{\Loc}{Loc}
\DeclareMathOperator{\Thick}{Thick}
\DeclareMathOperator{\supp}{supp}

\DeclareMathOperator{\Id}{Id}

\newcommand{\os}{\mathrm{os}}
\newcommand{\Z}{\mathbb Z}

\title[A Krause-type recollement for $\Q$-shaped derived categories]
{A Krause-type recollement for $\Q$-shaped derived categories}
\author[Y.J. Ma]{Yajun Ma}
\address{Yajun Ma: Department of Mathematics, Lanzhou Jiaotong University,
Lanzhou 730070, China}
\email{yjma@mail.lzjtu.cn}
\thanks{Yajun Ma is supported by NSF of China (Grant No. 12661004).}
\author[P. R. Yang]{Peiru Yang$^\ast$}
\address{Peiru Yang: School of Mathematics and Statistics,
Northeast Normal University, Changchun 130024, Jilin Province,
P. R. China}
\email{yangprlife@163.com}
\thanks{$^\ast$ Corresponding author.}

\makeatletter
\@namedef{subjclassname@2020}{\textup{2020} Mathematics Subject Classification}
\makeatother
\begin{document}
\begin{abstract}
Let $\kk$ be a commutative noetherian hereditary ring and let $\Q$
be a small, Hom-finite, locally bounded $\kk$-linear category
admitting a Serre functor and satisfying the strong retraction
property, with nilpotent pseudo-radical. For every $\kk$-algebra
$A$, we construct a recollement of the $\Q$-shaped homotopy category $\KI$ of objectwise injective $\Q$-shaped modules by its full subcategory
$\KAI$ of exact objects and the $\Q$-shaped derived category
$\DQ$. The construction uses
complete hereditary cotorsion pairs and does not require $A$
to be noetherian. If $A$ is left noetherian, both $\KI$ and
$\KAI$ are compactly generated by special preenvelopes of
finitely generated stalk modules. We identify $(\KAI)^c$ with the
idempotent completion of $(\KI)^c/L_I((\DQ)^c)$, where
$L_I$ is the fully faithful left adjoint of the localization.
This paper builds on and generalizes works of Krause and Liu.
\end{abstract}
\maketitle

\section{Introduction}

For a left noetherian ring $A$, Krause showed that the
canonical localization from the homotopy category of
complexes of injective $A$-modules to the derived category
fits into a recollement
\cite[Corollary~4.3]{Krause05}:
\[
\xymatrix@C=3.4pc{
 \mathcal K^{\mathrm{ac}}({}_A\Inj)\ar[r]^-{i_*}&
 \mathcal K({}_A\Inj)
 \ar@/_1.4pc/[l]_-{i^*}\ar@/^1.4pc/[l]^-{i^!}
 \ar[r]^-{Q}&
 \mathcal D(A)
 \ar@/_1.4pc/[l]_-{L}\ar@/^1.4pc/[l]^-{R}.}
\]
Here $\mathcal K^{\mathrm{ac}}({}_A\Inj)$ is the full
subcategory of acyclic complexes of injective modules,
called the stable derived category of $A$.
Both $\mathcal K({}_A\Inj)$ and
$\mathcal K^{\mathrm{ac}}({}_A\Inj)$ are compactly generated
\cite[Proposition~2.3 and Corollary~5.4]{Krause05}.
Moreover, the full subcategory of compact objects in
$\mathcal K^{\mathrm{ac}}({}_A\Inj)$ is triangle equivalent
to the idempotent completion of the singularity category
$\mathcal D^b({}_A\modu)/\mathcal K^b({}_A\prj)$
\cite[Corollary~5.4]{Krause05}.
Thus the recollement relates the derived category of $A$
to its singularity category through complexes of injective
modules.
Liu constructed the corresponding recollement for periodic
complexes and proved that all three categories in that
recollement are compactly generated
\cite[Theorem~3.1]{Liu23}.

Holm and J{\o}rgensen introduced the $\Q$-shaped derived
category as a generalization of the derived category of a ring.
Let $\kk$, $\Q$, and $A$ satisfy Setup~\ref{setup:standing},
and let $\Mod{\Q,A}$ denote the category of $\kk$-linear
functors $\Q\to\Mod{A}$.
They constructed projective and injective abelian model
structures on this functor category
\cite[Theorem~6.1]{HJ22}.
These model structures have the same weak equivalences
 and hence the same homotopy category, denoted by $\DQ$, which is called
the $\Q$-shaped derived category of $A$ \cite[Proposition~6.3]{HJ22}.
We write $\EE$ for their common class of trivial objects
and call these objects exact.
By \cite[Theorem~D]{HJ24}, the stalk modules
$S_q(A)$, with $q\in\Q$, form a set of compact generators
for $\DQ$.

The objectwise split exact structure on $\Mod{\Q,A}$
gives a corresponding homotopy category.
With this exact structure, $\Mod{\Q,A}$ is Frobenius,
and its stable category is the $\Q$-shaped homotopy
category $\KQ$ introduced by Slaftsos and Vit\'{o}ria
\cite[Definition~6.9 and Example~6.10]{SV26}.
They proved that the canonical functor
$I_*:\KQ\to\DQ$ is a Verdier localization with fully
faithful left and right adjoints
\cite[Theorem~6.11]{SV26}.
It therefore fits into a recollement
\[
\xymatrix@C=3.5pc{
 \KQA
 \ar[r]^-{j}
 &
 \KQ
 \ar@/_1.5pc/[l]_-{j_l}
 \ar@/^1.5pc/[l]^-{j_r}
 \ar[r]^-{I_*}
 &
 \DQ
 \ar@/_1.5pc/[l]_-{c}
 \ar@/^1.5pc/[l]^-{f}.}
\]
Here $\KQA=\Ker I_*$ is the full subcategory of $\KQ$
formed by the exact $\Q$-shaped modules, and $j$ is
the inclusion.
The functor $c$ is induced by cofibrant replacement
in the projective model structure, and $f$ is induced
by fibrant replacement in the injective model structure.

Throughout the paper, we assume Setup~\ref{setup:standing}.
Let $\I$ be the full subcategory of $\Mod{\Q,A}$ consisting
of objectwise injective $\Q$-shaped modules, and let $\KI$
be its image in $\KQ$.
We show that the inherited exact structure makes $\I$
a Frobenius category with stable category $\KI$.
The canonical localization restricts to a functor
$Q_I:\KI\to\DQ$ whose kernel is $\KAI$, the full subcategory
of exact objectwise injective $\Q$-shaped modules.
Our first main result places this functor in a recollement of triangulated categories.
Fibrant replacement provides a fully faithful right adjoint
to $Q_I$.
The key step is the construction of a fully faithful left
adjoint. To construct the left adjoint, we restrict the complete
hereditary cotorsion pair
$({}^{\perp_1}\EI,\EI)$ in $\Mod{\Q,A}$ to the Frobenius
category $\I$, where $\EI=\EE\cap\I$.
The resulting cotorsion pair induces the Bousfield localizing
pair $({}^\perp\KAI,\KAI)$ in $\KI$.
Consequently, $Q_I$ restricts to a triangle equivalence
${}^\perp\KAI\xrightarrow{\sim}\DQ$ (see Theorem \ref{thm:left-equivalence}).
A quasi-inverse, composed with the inclusion
${}^\perp\KAI\hookrightarrow\KI$, gives the fully faithful
left adjoint $L_I$ of $Q_I$.
Together with the right adjoint, this yields the following theorem; see Theorem \ref{thm:recollement}.
\begin{introthm}\label{intro:recollement}\label{thm:intro-main}
Under Setup~\ref{setup:standing}, there is a recollement of triangulated categories
\[
\xymatrix@C=3.5pc{
 \KAI
 \ar[r]^-{i_*}
 &
 \KI
 \ar@/_1.5pc/[l]_-{i^*}
 \ar@/^1.5pc/[l]^-{i^!}
 \ar[r]^-{Q_I}
 &
 \DQ
 \ar@/_1.5pc/[l]_-{L_I}
 \ar@/^1.5pc/[l]^-{R_I}.}
\]
 Thus
\[
 \KI/
 \KAI
 \simeq \DQ.
\]
\end{introthm}

For chain complexes and periodic complexes over a left
noetherian ring $A$, Theorem~\ref{thm:intro-main} recovers
the recollements of Krause and Liu, respectively;
see Example~\ref{exa:ordinary} or \cite{Krause05, Liu23}.
The recollement itself does not require $A$ to be left
noetherian. We impose this additional hypothesis below
to study compact generators and compact objects.

Assume that $A$ is left noetherian.
We use the complete hereditary cotorsion pairs
$({}^{\perp_1}\I,\I)$ and $({}^{\perp_1}\EI,\EI)$
in $\Mod{\Q,A}$ to construct compact generators for
$\KI$ and $\KAI$.
Completeness provides special preenvelopes
$M\to\rho M$ and $M\to\sigma M$, respectively.
Heredity and the Frobenius structures give natural
stable Hom--Ext isomorphisms.
For a noetherian object $M$, these isomorphisms show that
$\rho M$ and $\sigma M$ are compact, since
$\Ext^1_{\Q,A}(M,-)$ preserves coproducts.
To show that these objects generate $\KI$ and $\KAI$,
we use the fact that an object $X$ is injective if
$\Ext^1_{\Q,A}(M,X)=0$ for every noetherian object $M$.
Finite pseudo-radical filtrations then show that it suffices
to take special preenvelopes of the stalk modules $S_q(M)$,
where $q\in\Q$ and $M$ is a finitely generated left $A$-module.
The following theorem describes the resulting compact
generating sets.

\begin{introthm}\label{thm:B}
Assume Setup~\ref{setup:standing} and let $A$ be left noetherian.
Let $\mathcal M_A$ be a set of representatives of the isomorphism
classes of finitely generated left $A$-modules.
For each $q\in\Q$ and $M\in\mathcal M_A$, choose a special
$\I$-preenvelope $S_q(M)\to\rho S_q(M)$ and a special
$\EI$-preenvelope $S_q(M)\to\sigma S_q(M)$.
Then the following hold.
\begin{prt}
\item
The categories $\KI$ and $\KAI$ are compactly generated
by the sets $\mathcal G_I
 =\{\rho S_q(M)\mid q\in\Q,\ M\in\mathcal M_A\}$ and $
 \mathcal G_{\mathrm{ac}}
 =\{\sigma S_q(M)\mid q\in\Q,\ M\in\mathcal M_A\},$ respectively.
Moreover, $\sigma S_q(M)\cong i^*\rho S_q(M)$ in $\KAI$
for every $q\in\Q$ and $M\in\mathcal M_A$.

\item
The essential image $L_I((\DQ)^c)$ is a thick subcategory
of $(\KI)^c$, given by
\[
 L_I((\DQ)^c)
 =\Thick_{\KI}\{\rho S_q(A)\mid q\in\Q\}.
\]
The functor $i^*$ preserves compact objects and induces
a triangle equivalence up to direct summands
\[
 \overline{i^*}:
 (\KI)^c/L_I((\DQ)^c)
 \longrightarrow(\KAI)^c,
 \qquad C\longmapsto i^*C.
\]
Explicitly, $\overline{i^*}$ is fully faithful, and every
object of $(\KAI)^c$ is a direct summand of $i^*C$
for some $C\in(\KI)^c$.
\end{prt}
\end{introthm}

The proof of Theorem~\ref{thm:B} uses the stable Hom--Ext
isomorphisms and finite pseudo-radical filtrations of
noetherian $\Q$-shaped modules to obtain compact generators.
The description of compact objects follows from Neeman's
localization theorem \cite[Theorem~2.1]{Neeman92}.
For chain complexes, the quotient in Theorem~\ref{thm:B}
is triangle equivalent to the classical singularity category
$\mathcal D^b({}_A\modu)/\mathcal K^b({}_A\prj)$.
For general $\Q$, the recollement determines the corresponding
quotient through the subcategory $L_I((\DQ)^c)$.
Together with \cite[Theorem~D]{HJ24}, our results give
compact generators for all three terms of the recollement
and describe the relation between their compact objects.

Section~\ref{sec:preliminaries} fixes the standing hypotheses
and notation and recalls the $\Q$-shaped homotopy and
derived categories.
Section~\ref{sec:recollement} establishes the Frobenius
structures on $\I$ and $\EI$ and uses cotorsion pairs
to construct the adjoints of $Q_I$, proving the recollement
in Theorem~\ref{thm:intro-main}.
Section~\ref{sec:compact} constructs explicit compact
generators for $\KI$ and $\KAI$, identifies the essential
image $L_I((\DQ)^c)$, and proves the quotient description
of $(\KAI)^c$ in Theorem~\ref{thm:B}.
It also treats the case of finite left global dimension
and applies the results to chain complexes, periodic
complexes, and $N$-complexes.

\section{Preliminaries and notation}\label{sec:preliminaries}

Throughout this paper, $\kk$ is a commutative ring, $A$ is an associative $\kk$-algebra, and  $\Q$ is a small $\kk$-preadditive category.
We denote by $\Mod{A}$ the category of left $A$-modules and by $_{A}\mathsf{Prj}$ the subcategory of $\Mod{A}$ consisting of projective left $A$-modules. We work under the following standing assumptions.

\begin{setup}\label{setup:standing}
Throughout the paper, $\kk$ denotes a non-trivial
commutative noetherian hereditary ring and $A$ a unital
$\kk$-algebra. We assume
that the small $\kk$-pre-additive category $\Q$ satisfies the
following conditions; see \cite[Setup~2.5, Definition~7.3]{HJ22}.
\begin{enumerate}
\item[(1)] \emph{Hom-finiteness:} Each hom $\kk$-module $\Q(p,q)$
is finitely generated and projective.
\item[(2)] \emph{Local boundedness:} For each $q\in\Q$, the sets
\[
 N_-(q)=\{p\in\Q\mid\Q(p,q)\ne0\}\quad\text{and}\quad
 N_+(q)=\{p\in\Q\mid\Q(q,p)\ne0\}
\]
are finite.
\item[(3)] \emph{Existence of a Serre functor:} There is a
$\kk$-linear auto-equivalence $\mathbb S:\Q\to\Q$ and natural
isomorphisms
\[
 \Q(p,q)\cong\Hom_{\kk}(\Q(q,\mathbb S(p)),\kk).
\]
\item[(4)] \emph{Strong Retraction Property:} For any object $q$ in $\Q$, the unit map $\Bbbk\to\Q(q,q)$ (via $x\mapsto x\cdot \mathrm{id}_q$) has a $\Bbbk$-module retraction; whence there is a $\Bbbk$-module decomposition  $$\Q(q,q)=(\Bbbk\cdot{\rm id}_q)\oplus\mathfrak{r}_q$$ for $q\in \Q$ such that
\begin{itemize}
\item  $\mathfrak{r}_q\circ\mathfrak{r}_q\subseteq\mathfrak{r}_q$ for all $q$, and
\item $\Q(q,p)\circ \Q(p,q)\subseteq\mathfrak{r}_p$ for all $p\neq q$.
\end{itemize}
These submodules define the pseudo-radical ideal $\rQ$ by
$\rQ(q,q)=\mathfrak r_q$ and $\rQ(p,q)=\Q(p,q)$ for $p\ne q$.
\item[(5)] \emph{Nilpotence:} There is an integer $d\geq1$ such
that $\rQ^d=0$.
\end{enumerate}
All results are understood under these hypotheses. The additional
assumption that $A$ is left noetherian is imposed in
Section~\ref{sec:compact} and wherever explicitly stated.
\end{setup}

\begin{bfhpg}[\bf The category of $\Q$-shaped modules]
The category $\Mod{\Q,A}$ of $\kk$-linear functors from $\Q$ to the
 category $\Mod{A}$ of left $A$-modules is called the \emph{category of $\Q$-shaped $A$-modules}.
When $A=\kk$, we write $\Mod{\Q}$ instead of $\Mod{\Q,\kk}$.
The functor category $\Mod{\Q,A}$ is a Grothendieck category
with enough projectives.
Following \cite{HJ22}, we write ${}_{\Q,A}\Prj$ and
${}_{\Q,A}\Inj$ for its full subcategories of projective
and injective objects, respectively.
For left $A$-modules, the corresponding subcategories
are denoted by ${}_A\Prj$ and ${}_A\Inj$.
We write ${}_A\modu$ and ${}_A\mathrm{prj}$ for the full subcategories of finitely generated left $A$-modules and finitely generated projective left $A$-modules, respectively.
We use $\Hom_{\Q,A}$ and $\Ext^i_{\Q,A}$ for morphism groups
and Ext groups in $\Mod{\Q,A}$.

\begin{ipg}\label{adjoint triple}
 Recall from \cite[Corollary 3.9]{HJ22} that for every object $q$ in $\Q$ there is an adjoint triple $(F_{q},E_{q},G_{q})$ as follows:
  \begin{equation*}
  \xymatrix@C=4pc{
    \Mod{\Q,A}
    \ar[r]^-{E_{q}}
    &
    \Mod{A}
    \ar@/_1.8pc/[l]_-{F_{q}}
    \ar@/^1.8pc/[l]^-{G_{q}}
  }
  \qquad \text{given by} \qquad
  {\setlength\arraycolsep{1.5pt}
   \renewcommand{\arraystretch}{1.2}
  \begin{array}{rcl}
  F_q(M) &=& \Q(q,-) \otimes_\kk M \\
  E_q(X) &=& X\mspace{1.5mu}(q) \\
  G_{q}(M) &=& \Hom_{\kk}(\Q(-, q),M)\;
  \end{array}
  }
  \end{equation*}
for $M\in \Mod{A}$ and $X\in \Mod{\Q,A}$.
Here $E_{q}$ is the evaluation functor at $q.$
\end{ipg}

All three functors are exact
\cite[Corollary~3.9]{HJ22}.
Moreover, the objects $F_{q}(A)$, with $q\in\Q$, form a set
of finitely presentable projective generators of $\Mod{\Q,A}$
\cite[Proposition~3.12]{HJ22}.

The exact adjunctions yield the following isomorphisms; see \cite[Lemma 5.1]{HJ19}.

\begin{lem}\label{lem:ext-adjunction}
For $M\in \Mod{A}$, $X\in\Mod{\Q,A}$, and $n\geq0$, there are
natural isomorphisms
\[
\begin{aligned}
 \Ext^n_{\Q,A}(F_q(M),X)&\cong\Ext^n_A(M,X(q)),\\
 \Ext^n_{\Q,A}(X,G_q(M))&\cong\Ext^n_A(X(q),M).
\end{aligned}
\]
\end{lem}

Following \cite[Definition~7.9]{HJ22}, the left and right stalks
at $q\in\Q$ are
\[
 S\langle q\rangle=\Q(q,-)/\rQ(q,-)\in\Mod{\Q}
 \qquad\text{and}\qquad S\{q\}=\Q(-,q)/\rQ(-,q)\in\operatorname{Mod}_{\Q}.
\]
Recall from \cite[Proposition 7.15]{HJ22} that for every object $q$ in $\Q$, there is an adjoint triple $(C_{q},S_{q},K_{q})$ as follows:
  \begin{equation*}
  \xymatrix@C=4pc{
    \Mod{A}
    \ar[r]^-{S_{q}}
    &
    \Mod{\Q,A}
    \ar@/_1.8pc/[l]_-{C_{q}}
    \ar@/^1.8pc/[l]^-{K_{q}}
  }
  \qquad \text{given by} \qquad
  {\setlength\arraycolsep{1.5pt}
   \renewcommand{\arraystretch}{1.2}
  \begin{array}{rcl}
  C_q(X) &=& S\{q\} \otimes_\Q X \\
  S_q(M) &=& S\langle q\rangle\otimes_\kk M \\
  K_{q}(X) &=& \Hom_{\Q}(S\langle q\rangle,X)\;
  \end{array}
  }
  \end{equation*}
for $M\in \Mod{A}$ and $X\in \Mod{\Q,A}$.
In particular, the stalk functor $S_{q}$ is exact.
\end{bfhpg}

\begin{bfhpg}[\bf The Q-shaped homotopy category]
We first recall the terminology used in this construction.
Following \cite{Buhler10}, an exact category is an additive
category equipped with a class of kernel--cokernel pairs
satisfying Quillen's axioms.
Such a pair $X\xrightarrow{i}Y\xrightarrow{p}Z$ is called
a conflation; the morphisms $i$ and $p$ are called an
inflation and a deflation, respectively.
An exact category is Frobenius if it has enough projectives
and enough injectives, and these two classes coincide.
Their common objects are called \emph{projective-injective}.
The stable category has the same objects as the Frobenius
category. Its morphism groups are obtained by taking the
quotient modulo morphisms that factor through projective-injective
objects.
It has a canonical triangulated structure, with
distinguished triangles induced by conflations
\cite[Chapter~I, Section~2]{Happel88}.

To define the $\Q$-shaped homotopy category introduced by
Slaftsos and Vit\'{o}ria \cite{SV26}, we equip $\Mod{\Q,A}$ with the
objectwise split exact structure.
Its conflations are the short exact sequences
\[
 0\longrightarrow X'\xrightarrow{u}X
 \xrightarrow{v}X''\longrightarrow0
\]
for which the evaluated sequence
\[
 0\longrightarrow X'(q)\xrightarrow{u(q)}X(q)
 \xrightarrow{v(q)}X''(q)\longrightarrow0
\]
splits in $\Mod{A}$ for every $q\in\Q$.
These splittings are not required to be natural in $q$,
so the original sequence need not split in $\Mod{\Q,A}$.
We denote the associated exact category by
$\Mod{\Q,A}_{\mathrm{os}}$.

Under the standing hypotheses, $\Mod{\Q,A}_{\mathrm{os}}$
is Frobenius \cite[Example~4.18(c)]{SV26}.
Its projective-injective objects are called contractible; their full subcategory is denoted by $\operatorname{Ctr}_{\Q}(A).$ Further characterizations are given in Subsection~\ref{subsec:contractible}.
The stable category is the \emph{$\Q$-shaped homotopy category},
denoted by $\KQ$; see \cite[Definition~6.9]{SV26}.
Thus \[
 \Hom_{\KQ}(X,Y)
 =
 \Hom_{\Q,A}(X,Y)/
 [\operatorname{Ctr}_{\Q}(A)](X,Y),
\]
where $[\operatorname{Ctr}_{\Q}(A)](X,Y)$ consists of the
morphisms that factor through a contractible object.
Distinguished triangles in $\KQ$ are induced, up to isomorphism, by objectwise split
short exact sequences.
By \cite[Construction 5.6 and Proposition 5.7]{HJ24} there are objectwise split sequences
\begin{align*}
 0&\longrightarrow\mathbb K(X)\longrightarrow\mathbb F(X)
 \xrightarrow{\epsilon_X}X\longrightarrow0,\\
 0&\longrightarrow X\xrightarrow{\eta_X}\mathbb G(X)
 \longrightarrow\mathbb C(X)\longrightarrow0,
\end{align*}
where \begin{equation*}
 \mathbb F(X)=\bigoplus_{q\in\Q}F_q(X(q))
 \qquad\text{and}\qquad \mathbb G(X)=\prod_{q\in\Q}G_q(X(q)).
\end{equation*}
and $\epsilon$ and $\eta$ are the unique natural transformations defined by the universal properties of coproducts and products such that $\epsilon_{q}^{X}$ is the counit of $(F_{q},E_{q})$ at $X$ and $\eta_{q}^{X}$ is the unit of $(E_{q},G_{q})$ at $X$, respectively.
The remaining objects in the sequences are defined by $\mathbb C(X)=\Coker \eta_{X}$ and $\mathbb K(X)=\Ker \epsilon_{X}$.

These sequences above provide enough projectives and
enough injectives for $\Mod{\Q,A}_{\mathrm{os}}$. They also identify
suspension and inverse suspension in $\KQ$:
\[
 \Sigma X\cong\mathbb C(X)
 \qquad\text{and}\qquad
 \Sigma^{-1}X\cong\mathbb K(X)
 \quad\text{in }\KQ.
\]
\end{bfhpg}

\begin{bfhpg}[\bf The Q-shaped derived category ]
We next recall the abelian model structures on $\Mod{\Q,A}$
whose common homotopy category is the $\Q$-shaped derived
category $\DQ$. Their construction uses cotorsion pairs
with respect to the abelian exact structure, for which we
fix the following terminology.

\begin{dfn}\label{dfn:cotorsion}
Let $\mathcal C$ be an exact category, and let
$\Ext^n_{\mathcal C}$ denote the Yoneda extension groups
with respect to its exact structure. For a class
$\mathcal X$ of objects of $\mathcal C$, set
\[
\begin{aligned}
 {}^{\perp_1}\mathcal X
 &=
 \{U\in\mathcal C\mid
   \Ext^1_{\mathcal C}(U,X)=0
   \text{ for all }X\in\mathcal X\},\\
 \mathcal X^{\perp_1}
 &=
 \{V\in\mathcal C\mid
   \Ext^1_{\mathcal C}(X,V)=0
   \text{ for all }X\in\mathcal X\}.
\end{aligned}
\]
A pair $(\mathcal U,\mathcal V)$ of classes of objects is
a \emph{cotorsion pair} in $\mathcal C$ if
$\mathcal U={}^{\perp_1}\mathcal V$ and
$\mathcal V=\mathcal U^{\perp_1}$.
It is \emph{complete} if every $X\in\mathcal C$ admits
conflations
\[
 0\longrightarrow V_X\longrightarrow U_X
 \longrightarrow X\longrightarrow0
 \qquad\text{and}\qquad
 0\longrightarrow X\longrightarrow V^X
 \longrightarrow U^X\longrightarrow0
\]
with $U_X,U^X\in\mathcal U$ and
$V_X,V^X\in\mathcal V$.
The morphisms $U_X\longrightarrow X$ and
$X\longrightarrow V^X$ are called a special
$\mathcal U$-precover and a special
$\mathcal V$-preenvelope, respectively.
The pair is \emph{hereditary} if
$\Ext^n_{\mathcal C}(U,V)=0$ for all $n\geq1$,
$U\in\mathcal U$, and $V\in\mathcal V$.
We say that it is \emph{generated by a set}
$\mathcal S\subseteq\mathcal C$ if
$\mathcal V=\mathcal S^{\perp_1}$.
\end{dfn}

Following \cite[Definition~4.1]{HJ22}, we define
\[
 \EE=\{X\in\Mod{\Q,A}\mid X^\natural
       \text{ has finite projective dimension in }\Mod{\Q}\},
\]
where $(-)^\natural$ is the forgetful functor.
Objects of $\EE$ are called \emph{exact}, and they are precisely the objects whose
$\Q$-shaped homology, equivalently cohomology, vanishes
in all positive degrees \cite[Theorem~7.1]{HJ22}.
In particular, by the proof of \cite[Theorem D]{HJ24}, one has $X\in\EE$ if and only if $\Ext^1_{\Q,A}(S_q(A),X)=0$ for every $q\in\Q.$
By \cite[Theorems~4.4 and~6.1]{HJ22}, the class $\EE$ is
thick, and the hereditary Hovey triples
\[
 ({}^{\perp_1}\EE,\EE,\Mod{\Q,A})
 \qquad\text{and}\qquad
 (\Mod{\Q,A},\EE,\EE^{\perp_1})
\]
define the projective and injective abelian model structures,
respectively. These model structures have the same weak
equivalences \cite[Proposition~6.3]{HJ22}.
Their common homotopy category is the
\emph{$\Q$-shaped derived category of $A$}, denoted by $\DQ$.
\end{bfhpg}

To describe the relation between $\KQ$ and $\DQ$,
we recall the definition of a recollement of triangulated categories.

\begin{dfn}
A \emph{recollement} of triangulated categories is a diagram
of exact functors
\[
\xymatrix@C=3.5pc{
 \mathcal T'
 \ar[r]^-{i_*}
 &
 \mathcal T
 \ar@/_1.5pc/[l]_-{i^*}
 \ar@/^1.5pc/[l]^-{i^!}
 \ar[r]^-{j^*}
 &
 \mathcal T''
 \ar@/_1.5pc/[l]_-{j_!}
 \ar@/^1.5pc/[l]^-{j_*}
}
\]
such that:
\begin{enumerate}
\item
$(i^*,i_*,i^!)$ and $(j_!,j^*,j_*)$ are adjoint triples;
\item
$i_*$, $j_!$, and $j_*$ are fully faithful;
\item
$\mathrm{Im} i_*=\Ker j^{*}.$
\end{enumerate}
\end{dfn}

Under Setup~\ref{setup:standing}, by \cite[Theorem~6.11]{SV26},  there is a recollement of triangulated categories
\[
\xymatrix@C=3.5pc{
 \mathcal K_{\Q}^{\mathrm{ac}}(A)
 \ar[r]^-{j}
 &
 \KQ
 \ar@/_1.5pc/[l]_-{j_l}
 \ar@/^1.5pc/[l]^-{j_r}
 \ar[r]^-{I_*}
 &
 \DQ
 \ar@/_1.5pc/[l]_-{c}
 \ar@/^1.5pc/[l]^-{f}.
}
\]
Here $I_*$ is the canonical localization,
$\mathcal K_{\Q}^{\mathrm{ac}}(A)=\Ker I_*$, and $j$ is
the inclusion. The fully faithful left adjoint $c$ is
induced by cofibrant replacement in the projective model
structure $({}^{\perp_1}\EE,\EE,\Mod{\Q,A})$, while the fully faithful right adjoint $f$
is induced by fibrant replacement in the injective model
structure $(\Mod{\Q,A},\EE,\EE^{\perp_1})$. In Section~\ref{sec:recollement}, we study the restriction
$Q_I=I_*|_{\KI}:\KI\longrightarrow\DQ$ and construct the
corresponding recollement with middle term $\KI$.

\section{A recollement for $\Q$-shaped homotopy categories of injectives}
\label{sec:recollement}

Under Setup~\ref{setup:standing}, let $Q_I:\KI\to\DQ$
denote the restriction of the canonical localization
$I_*:\KQ\to\DQ$. Its kernel is $\KAI$.
We construct fully faithful left and right adjoints to $Q_I$
and thereby obtain a recollement relating $\KAI$, $\KI$
and $\DQ$. No noetherian hypothesis on $A$ is required.

\subsection{Contractible objects and the Frobenius structure on $\I$}
\label{subsec:contractible}
We first describe the contractible objects and then establish the Frobenius structure on the full subcategory $\I=\{X\in \Mod{\Q,A} \mid X(q)\in {}_A\Inj\}$.
This identifies its stable category with $\KI$.

\begin{prp}\label{prp:contractible}
Let $\Mod{\Q,A}_{\os}$ denote $\Mod{\Q,A}$ endowed with the
objectwise split exact structure. For an object
$X\in\Mod{\Q,A}$, the following statements are equivalent.
\begin{eqc}
\item $X$ is projective-injective in $\Mod{\Q,A}_{\os}$.
\item There exists a family $(M_q)_{q\in\Q}$ of left
$A$-modules such that
\[
 X\cong\prod_{q\in\Q}G_q(M_q).
\]
\item There is an isomorphism
\[
 X\cong\prod_{q\in\Q}G_qK_q(X).
\]
\item There exists a family $(N_q)_{q\in\Q}$ of left
$A$-modules such that
\[
 X\cong\bigoplus_{q\in\Q}F_q(N_q).
\]
\item There is an isomorphism
\[
 X\cong\bigoplus_{q\in\Q}F_qC_q(X).
\]
\item For every $q\in\Q$ and every $M\in\Mod{A}$,
\[
 \Ext^1_{\Mod{\Q,A}_{\os}}(S_q(M),X)=0.
\]
\item For every $q\in\Q$ and every $M\in\Mod{A}$,
\[
 \Ext^1_{\Mod{\Q,A}_{\os}}(X,S_q(M))=0.
\]
\end{eqc}
\end{prp}

\begin{proof}
By \cite[Remark~4.8, Example~4.18(c)]{SV26}, the exact category
$\Mod{\Q,A}_{\os}$ is Frobenius. Thus its projective and
injective objects coincide.

Since every $A$-module is injective in the split exact structure,
it follows from \cite[Proposition~4.19]{SV26} that
$\textup{(i)}\Leftrightarrow\textup{(vi)}$, and its proof yields
\[
 X\cong\prod_{q\in\Q}G_qK_q(X)
\]
whenever (i) holds. Hence $\textup{(i)}\Rightarrow\textup{(iii)}\Rightarrow\textup{(ii)}$.
Conversely, the adjunction $(E_q,G_q)$ shows that $G_q(M)$
is injective in $\Mod{\Q,A}_{\os}$ for every $A$-module $M$.
Since products of injectives are injective, one has $\textup{(ii)}\Rightarrow\textup{(i)}$.

By \cite[Proposition~3.7 and Lemma~3.4]{HJ24}, one has
\[
 \prod_{q\in\Q}G_q(M_q)
 \cong\bigoplus_{q\in\Q}G_q(M_q)
 \cong\bigoplus_{q\in\Q}F_{\mathbb S^{-1}(q)}(M_q),
\]
so $\textup{(ii)}\Rightarrow\textup{(iv)}$. The adjunction $(F_q,E_q)$ shows
that each $F_q(N)$ is projective in $\Mod{\Q,A}_{\os}$.
Thus $\textup{(iv)}\Rightarrow\textup{(i)}$.
Moreover, by \cite[Lemma~7.28]{HJ22} we have natural isomorphisms
\begin{equation*}
 C_pF_q(M)\cong
 \begin{cases}
 M,&p=q,\\
 0,&p\ne q.
 \end{cases}
\end{equation*}
Since $C_p$ preserves direct sums, the decomposition in (iv)
implies $C_p(X)\cong N_p$, proving $\textup{(iv)}\Rightarrow\textup{(v)}$.
The converse is immediate.

\smallskip
\noindent

The implication $\textup{(i)}\Rightarrow\textup{(vii)}$ follows from projectivity.
Assume (vii). We first show that each canonical quotient
$\pi_q:X(q)\to C_q(X)$ splits. Apply
$\Hom_{\Q,A}(-,S_q(M))$ to the canonical objectwise split
exact sequence
\[
 0\longrightarrow\mathbb K(X)\xrightarrow{a}
 \mathbb F(X)\xrightarrow{\epsilon_X}X\longrightarrow0.
\]
By (vii) and the adjunction $(C_q,S_q)$, the morphism
\[
 \Hom_A(C_q\mathbb F(X),M)
 \longrightarrow\Hom_A(C_q\mathbb K(X),M)
\]
is surjective for every $M$. Taking $M=C_q\mathbb K(X)$
shows that $C_q(a)$ is a split monomorphism. Right exactness
of $C_q$ and \cite[Lemma~7.28]{HJ22} therefore give a split
exact sequence
\[
 0\longrightarrow C_q\mathbb K(X)\longrightarrow X(q)
 \xrightarrow{\pi_q}C_q(X)\longrightarrow0.
\]
Choose sections $s_q$ of $\pi_q$. By the adjunctions
$(F_q,E_q)$, they induce a morphism
\[
 \psi:Y:=\bigoplus_{q\in\Q}F_qC_q(X)\longrightarrow X
\]
such that $C_q(\psi)=\pi_qs_q=\mathrm{Id}_{C_q(X)}$
under the natural identification $C_q(Y)\cong C_q(X)$ supplied by \cite[Lemma~7.28]{HJ22}.
Since $C_{q}$ is right exact for any $q\in\Q$, we have that  $C_q(\Coker\psi)=\Coker C_{q}(\psi)=0$ for all $q$.
It follows from \cite[Proposition 7.19]{HJ22} that
 $\psi$ is an epimorphism.

It remains to show that $\Ker\psi=0$. Let $J$ be an injective
cogenerator of $\Mod{A}$.
Every extension of $X$ by $S_q(J)$ is objectwise split.
Consequently, (vii) and the adjunction $(C_q,S_q)$ give
\[
 0=\Ext^1_{\Mod{\Q,A}_{\os}}(X,S_q(J))
  =\Ext^1_{\Q,A}(X,S_q(J))
  \cong\Hom_A(L_1C_q(X),J).
\]
Here the last isomorphism follows by applying the adjunction
to a projective resolution of $X$ and using the exactness
of $\Hom_A(-,J)$. Since $J$ is a cogenerator in $\Mod{A}$, one has
$L_1C_q(X)=0$ for every $q$.
Writing $V=\Ker\psi$. Then we have the following sequence
\[
 0\longrightarrow C_q(V)\longrightarrow C_q(Y)
 \xrightarrow{C_q(\psi)}C_q(X)\longrightarrow0.
\]
As $C_q(\psi)$ is an isomorphism, $C_q(V)=0$ for every $q$,
and hence $V=0$ by \cite[Proposition 7.19]{HJ22}. Thus $\psi$ is an isomorphism, proving
$\textup{(vii)}\Rightarrow\textup{(v)}$.
\end{proof}

\begin{dfn}\label{dfn:contractible-Q-module}
An object $X\in\Mod{\Q,A}$ is called \emph{contractible} if it
satisfies the equivalent conditions of Proposition~\ref{prp:contractible}. We denote the full subcategory of contractible objects by $\operatorname{Ctr}_{\Q}(A)$.
\end{dfn}

\begin{rmk}\label{rmk:contractible-complexes}
Let $\Q$ be the $\Bbbk$-linear path category of the quiver
\[
 \cdots\longrightarrow q+1\longrightarrow q
 \longrightarrow q-1\longrightarrow\cdots
\]
modulo the relations that all paths of length two vanish.
Then $\Mod{\Q,A}$ identifies with the category of chain
complexes of left $A$-modules, and
\[
 K_q(X)\cong\Ker(d_q^X)
 \qquad\text{and}\qquad C_q(X)\cong\Coker(d_{q+1}^X).
\]
Moreover, $F_q(M)$ and $G_q(M)$ are shifts of disk complexes.
Therefore a complex $X$ is contractible if and only if it admits
a decomposition into disk complexes:
\[
 X\cong\bigoplus_{q\in\Z}F_qC_q(X)
 \cong\prod_{q\in\Z}G_qK_q(X).
\]
Thus Definition \ref{dfn:contractible-Q-module} recovers the usual notion of a
contractible complex.
\end{rmk}

\begin{cor}\label{cor:contractible-stable-zero}
For an object $X\in\Mod{\Q,A}$, the following conditions
are equivalent:
\begin{eqc}
\item $X$ is contractible;
\item $X\cong0$ in $\KQ$;
\item $\mathrm{Id}_X$ factors through a contractible object.
\end{eqc}
The class $\operatorname{Ctr}_{\Q}(A)$ of contractible
objects is closed under direct summands, and
\[
 \KQ\simeq\Mod{\Q,A}_{\os}/[\operatorname{Ctr}_{\Q}(A)],
\]
where $[\operatorname{Ctr}_{\Q}(A)]$ denotes the ideal
of morphisms factoring through contractible objects.
\end{cor}

\begin{proof}
By Proposition~\ref{prp:contractible}, the contractible
objects are precisely the projective-injective objects of
$\Mod{\Q,A}_{\os}$. In particular, they are closed
under direct summands.
Since $\KQ$ is the stable category of this Frobenius
category, the quotient description follows.
Moreover, $X\cong0$ in $\KQ$ if and only if
$\mathrm{Id}_X$ belongs to $[\operatorname{Ctr}_{\Q}(A)]$.
This proves $\textup{(ii)}\Leftrightarrow\textup{(iii)}$.

The implication $\textup{(i)}\Rightarrow\textup{(iii)}$ follows by factoring
$\mathrm{Id}_X$ through $X$ itself.
Conversely, assume that there are morphisms
$u:X\to C$ and $v:C\to X$, with $C$ contractible,
such that $vu=\mathrm{Id}_X$.
Then $v$ is a split epimorphism with section $u$, and hence
$C\cong X\oplus\Ker v$ in $\Mod{\Q,A}$.
Since contractible objects are closed under direct summands,
$X$ is contractible. This proves $\textup{(iii)}\Rightarrow\textup{(i)}$.
\end{proof}

Recall that
\[
 \I=
 \{X\in\Mod{\Q,A}\mid X(q)\in{}_A\Inj
   \text{ for every }q\in\Q\}.
\]
Since $E_q$ is right adjoint to the exact functor $F_q$,
it preserves injective objects. Hence
${}_{\Q,A}\Inj\subseteq\I$.

\begin{lem}\label{lem:canonical-I}
For every $X\in\I$, one has
$\mathbb F(X),\mathbb G(X)\in{}_{\Q,A}\Inj$ and
$\mathbb K(X),\mathbb C(X)\in\I$.
\end{lem}

\begin{proof}
Since the evaluation functor $E_q$ is exact, its right
adjoint $G_q$ preserves injective objects.
For $X\in\I$, each $X(q)$ is an injective $A$-module.
Thus $\mathbb G(X)=\prod_{q\in\Q}G_q(X(q))$
is injective in $\Mod{\Q,A}$.
By \cite[Lemma~3.4 and Proposition~3.7]{HJ24},
\[
 \mathbb F(X)
 =\bigoplus_{q\in\Q}F_q(X(q))
 \cong\bigoplus_{q\in\Q}G_{\mathbb S(q)}(X(q))
 \cong\prod_{q\in\Q}G_{\mathbb S(q)}(X(q)).
\]
Thus
$\mathbb F(X)$ is also injective in $\Mod{\Q,A}$.

For every $p\in\Q$, the functor $E_p$ preserves
injective objects because its left adjoint $F_p$ is exact.
Consequently, $\mathbb F(X)(p)$ and $\mathbb G(X)(p)$
are injective $A$-modules.
The canonical objectwise split exact sequences
\[
 0\longrightarrow\mathbb K(X)
 \longrightarrow\mathbb F(X)
 \longrightarrow X\longrightarrow0
\]
and
\[
 0\longrightarrow X
 \longrightarrow\mathbb G(X)
 \longrightarrow\mathbb C(X)\longrightarrow0
\]
show that $\mathbb K(X)(p)$ and $\mathbb C(X)(p)$
are direct summands of $\mathbb F(X)(p)$ and
$\mathbb G(X)(p)$, respectively.
They are therefore injective, proving that
$\mathbb K(X),\mathbb C(X)\in\I$.
\end{proof}

\begin{prp}\label{prp:I-Frobenius}
The subcategory $\I$ is
extension-closed in $\Mod{\Q,A}$, and the exact structures on $\I$ induced by the abelian and objectwise split exact structures on $\Mod{\Q,A}$
coincide. With this
exact structure, $\I$ is Frobenius and
\[
 \Prj(\I)=\Inj(\I)={}_{\Q,A}\Inj.
\]
\end{prp}

\begin{proof}
Let $0\to X\to Y\to Z\to0$ be a short exact sequence
in $\Mod{\Q,A}$ with $X,Z\in\I$.
For every $q\in\Q$, the evaluated sequence splits
because $X(q)$ is injective. Hence
$Y(q)\cong X(q)\oplus Z(q)$ is injective.
This proves that $\I$ is extension-closed and that
the two induced exact structures coincide.

For each $X\in\I$, by Lemma~\ref{lem:canonical-I}
there are conflations
\[
 0\longrightarrow\mathbb K(X)
 \longrightarrow\mathbb F(X)
 \longrightarrow X\longrightarrow0
\]
and
\[
 0\longrightarrow X
 \longrightarrow\mathbb G(X)
 \longrightarrow\mathbb C(X)\longrightarrow0
\]
in $\I$ with $\mathbb F(X),\mathbb G(X)\in{}_{\Q,A}\Inj$.
By Proposition~\ref{prp:contractible}, $\mathbb F(X)$ and $\mathbb G(X)$
are projective-injective in $\Mod{\Q,A}_{\os}$.
Thus $\I$ has enough projectives and enough injectives.

If $P$ is projective in $\I$, the deflation
$\mathbb F(P)\to P$ splits, so $P$ is a direct summand
of $\mathbb F(P)$.
Similarly, every injective object $J$ of $\I$ is a
direct summand of $\mathbb G(J)$.
Hence every projective or injective object of $\I$ belongs
to ${}_{\Q,A}\Inj$ by Lemma \ref{lem:canonical-I}.

Conversely, let $J\in{}_{\Q,A}\Inj$.
Then $J\in\I$.
The canonical monomorphism $J\to\mathbb G(J)$ splits
in $\Mod{\Q,A}$ by injectivity of $J$.
Consequently, $J$ is a direct summand of the projective-injective
object $\mathbb G(J)$ and is therefore projective-injective
in $\I$.
This proves the asserted equality and shows that $\I$
is Frobenius.
\end{proof}

\begin{lem}\label{lem:factor-ideal}
Let $X$ and $Y$ be in $\I$. Then a morphism $X\to Y$ factors through a contractible object of $\Mod{\Q,A}$ if and only if it factors through an object of ${}_{\Q,A}\Inj$.
\end{lem}
\begin{proof}
Since every object of ${}_{\Q,A}\Inj$ is contractible by the proof
of Proposition~\ref{prp:I-Frobenius},  a factorization through
an object of ${}_{\Q,A}\Inj$ is a factorization through a contractible object.
Conversely, suppose that $u:X\to Y$ admits a factorization
$u=ba$ with $a:X\to C$ and $b:C\to Y$ where $C$ is contractible.
Note that there is an inflation $$0\to X\to \mathbb{G}(X)\to \mathbb{C}(X)\to 0$$ in $\Mod{\Q,A}_{\os}$.
Since $C$ is an injective object in $\Mod{\Q,A}_{\os}$ by Proposition \ref{prp:contractible}, there is a morphism
$\widetilde a:\mathbb G(X)\to C$ with
$\widetilde a\eta_X=a$. Therefore $
 u=(b\widetilde a)\eta_X.$
This is the required factorization, since $\mathbb G(X)$ is
injective in $\Mod{\Q,A}$ by Lemma~\ref{lem:canonical-I}.
\end{proof}

\begin{prp}\label{prp:stable-I}
The subcategory $\mathcal K_{\Q}({}_A\Inj)$ of
$\mathcal K_{\Q}(A)$ is a triangulated category, and
the inclusion $\I\hookrightarrow\Mod{\Q,A}_{\os}$ induces
a triangle equivalence
\[
 \underline{\I}\xrightarrow{\sim}\mathcal K_{\Q}({}_A\Inj).
\]
\end{prp}

\begin{proof}
By Proposition~\ref{prp:I-Frobenius}, the inclusion is exact
and preserves projective-injective objects, hence induces
a triangle functor $\iota:\underline{\I}\longrightarrow\mathcal K_{\Q}(A).$
By Lemma \ref{lem:factor-ideal}, the induced functor $\iota$ is fully faithful, and its essential image is exactly $\mathcal K_{\Q}({}_A\Inj)$.
Thus $\mathcal K_{\Q}({}_A\Inj)$ is triangulated, and the asserted equivalence follows.
\end{proof}

\begin{lem}\label{lem:Ext-comparison}
For $X,Y\in\I$, there are natural isomorphisms
\[
 \Ext^1_{\Q,A}(X,Y)
 \cong\Ext^1_{\I}(X,Y)
 \cong\Hom_{\KI}(X,\Sigma Y).
\]
\end{lem}

\begin{proof}
The first isomorphism follows because $\I$ is a full extension-closed subcate-
gory of $\Mod{\Q,A}$ with the inherited exact structure; see Proposition~\ref{prp:I-Frobenius}.

Consider the conflation $0\to Y\to \mathbb{G}(Y)\xrightarrow{p} \mathbb{C}(Y)\to0$ in $\I$.
Since $\mathbb{G}(Y)$ is projective-injective in $\I$, applying $\Hom_{\I}(X,-)$ gives an exact sequence
\[
 \Hom_{\I}(X,\mathbb{G}(Y))\xrightarrow{p_*}\Hom_{\I}(X,\mathbb{C}(Y))
 \longrightarrow\Ext^1_{\I}(X,Y)\longrightarrow0.
\]
It is straightforward to check that a morphism $X\to \mathbb{C} (Y)$ factors through a projective-injective
object if and only if it factors through $p$ as every
morphism from a projective object to $\mathbb{C}(Y)$ lifts along $p$.
We therefore obtain natural isomorphisms
\[
 \Hom_{\underline{\I}}(X,\Sigma Y)\cong\Hom_{\underline{\I}}(X,\mathbb{C}(Y))\cong\Ext^1_{\I}(X,Y),
\]
as desired.
\end{proof}

Let $j_{\mathrm{inj}}:\KI\hookrightarrow\KQ$ be the inclusion,
and set $Q_I=I_*j_{\mathrm{inj}}:\KI\to\DQ$.
The right adjoint $f$ of $I_*$ is induced by fibrant replacement
in the injective model structure $(\Mod{\Q,A},\EE,\EE^{\perp_1})$.
Since its fibrant objects are objectwise injective by
\cite[Proposition~3.2]{HJ24}, $f$ factors through $j_{\mathrm{inj}}$.
Thus we obtain a functor $R_I:\DQ\to\KI$ with
$j_{\mathrm{inj}}R_I=f$, which is right adjoint to $Q_I$
as shown below.

\begin{prp}\label{prp:right-adjoint}
There is an adjunction
\[
 \xymatrix@C=4pc{
 \KI\ar@<0.7ex>[r]^{Q_I}&
 \DQ\ar@<0.7ex>[l]^{R_I},}
\]
where the upper arrow is the left adjoint, and the functor $R_I$ is fully
faithful.
\end{prp}
\begin{proof}
The adjunction $(I_*,f)$ and the full faithfulness of
$j_{\mathrm{inj}}$ give natural isomorphisms
\[
\begin{aligned}
 \Hom_{\DQ}(Q_IX,Y)
 &=\Hom_{\DQ}(I_*j_{\mathrm{inj}}X,Y)\\
 &\cong\Hom_{\KQ}(j_{\mathrm{inj}}X,fY)\\
 &=\Hom_{\KQ}(j_{\mathrm{inj}}X,j_{\mathrm{inj}}R_IY)\\
 &\cong\Hom_{\KI}(X,R_IY)
\end{aligned}
\]
for $X\in\KI$ and $Y\in\DQ$, where we use
$Q_I=I_*j_{\mathrm{inj}}$ and $f=j_{\mathrm{inj}}R_I$.
Thus $(Q_I,R_I)$ is an adjoint pair.
Its counit at $Y$ identifies with the counit
\[
 Q_IR_IY=I_*fY\longrightarrow Y
\]
of $(I_*,f)$.
Since $f$ is fully faithful, this morphism is an
isomorphism. Hence $R_I$ is fully faithful.
\end{proof}

\begin{dfn}\label{dfn:acyclic-I}
Set $\EI=\EE\cap\I$.
The full subcategory of $\KI$ consisting of the objects
of $\EI$ is denoted by $\KAI$.
Thus its objects are the exact objectwise injective
$\Q$-shaped modules, and
\[
 \KAI=\Ker Q_I.
\]
\end{dfn}
Since $Q_I$ is a triangle functor, $\KAI$ is a thick
subcategory of $\KI$.
We next identify $\KAI$ with the stable category of $\EI$.

\begin{prp}\label{prp:EI-Frobenius}
With the exact structure inherited from $\I$, the category
$\EI$ is Frobenius, and
\[
 \Prj(\EI)=\Inj(\EI)={}_{\Q,A}\Inj.
\]
The inclusion $\EI\hookrightarrow\Mod{\Q,A}_{\os}$
induces a triangle equivalence
\[
 \underline{\EI}\xrightarrow{\ \sim\ }\KAI.
\]
\end{prp}

\begin{proof}
By \cite[Theorem~4.4]{HJ22}, the class $\EE$ is thick
and contains ${}_{\Q,A}\Inj$.
Hence $\EI$ is extension-closed in $\I$ and contains
all its projective-injective objects.

For $X\in\EI$, by Lemma~\ref{lem:canonical-I} one has
conflations
\[
 0\to\mathbb K(X)\to\mathbb F(X)\to X\to0
 \qquad\text{and}\qquad
 0\to X\to\mathbb G(X)\to\mathbb C(X)\to0
\]
in $\I$ with
$\mathbb F(X),\mathbb G(X)\in{}_{\Q,A}\Inj$.
Since $\EE$ is thick, $\mathbb K(X)$ and $\mathbb C(X)$
also belong to $\EE$.
Thus both sequences are conflations in $\EI$, showing
that $\EI$ has enough projectives and injectives.
Every projective or injective object of $\EI$ is a direct
summand of one of these middle terms and therefore
belongs to ${}_{\Q,A}\Inj$.
This proves the asserted equality and shows that $\EI$ is a Frobenius category.

By Lemma~\ref{lem:factor-ideal}, the inclusion
$\EI\hookrightarrow\Mod{\Q,A}_{\os}$ induces a fully
faithful triangle functor
$\underline{\EI}\to\KQ$.
Its essential image is $\KAI$ by definition, yielding
the required triangle equivalence.
\end{proof}

For a class $\mathcal X$ of objects in a
triangulated category $\mathcal T$, set
\[
\begin{aligned}
 {}^\perp\mathcal X
 &=
 \{Y\in\mathcal T\mid
   \Hom_{\mathcal T}(Y,\Sigma^nX)=0
   \text{ for all }X\in\mathcal X,\ n\in\mathbb Z\},\\
 \mathcal X^\perp
 &=
 \{Y\in\mathcal T\mid
   \Hom_{\mathcal T}(X,\Sigma^nY)=0
   \text{ for all }X\in\mathcal X,\ n\in\mathbb Z\}.
\end{aligned}
\]

\begin{prp}\label{prp:right-orthogonal}
The functor $R_I$ induces a triangle equivalence
\[
 R_I:\DQ\xrightarrow{\sim}(\KAI)^\perp.
\]
More precisely, an object $X\in\KI$ belongs to $(\KAI)^\perp$
if and only if the unit $\eta_X:X\to R_IQ_IX$ is an isomorphism.
\end{prp}

\begin{proof}
For $Z\in\KAI$, $Y\in\DQ$, and $n\in\Z$, adjunction gives
\[
 \Hom_{\KI}(Z,\Sigma^nR_IY)
 \cong\Hom_{\DQ}(Q_IZ,\Sigma^nY)=0.
\]
Thus $\operatorname{Im}R_I\subseteq(\KAI)^\perp$.

Conversely, let $X\in(\KAI)^\perp$ and complete the unit
to a triangle in $\KI$
\begin{equation*}
 X\xrightarrow{\eta_X}R_IQ_I(X)\longrightarrow C_X
 \longrightarrow\Sigma X.
\end{equation*}
Note that $R_{I}$ is fully faithful. It follows that $Q_I(\eta_X)$
is an isomorphism. Hence $Q_I(C_X)=0$, and so $C_X\in\KAI$.
Since $X$ and $R_IQ_I(X)$ belong to $(\KAI)^\perp$,
so does $C_X$. Therefore
$\Hom_{\KI}(C_X,C_X)=0$, forcing $C_X\cong0$.
Thus $\eta_X$ is an isomorphism.
Hence the full faithfulness of $R_I$ gives the asserted
triangle equivalence.
\end{proof}

To construct a left adjoint of $Q_I$, we use special
precover sequences to show that $Q_I$ restricts to a
triangle equivalence ${}^\perp\KAI\to\DQ$.
For this purpose, set
\[
\begin{aligned}
 \WI&={}^{\perp_1}\I
 =\{X\in\Mod{\Q,A}\mid
   \Ext^1_{\Q,A}(X,I)=0\text{ for every }I\in\I\},\\
 \UU&={}^{\perp_1}\EI
 =\{X\in\Mod{\Q,A}\mid
   \Ext^1_{\Q,A}(X,E)=0\text{ for every }E\in\EI\}.
\end{aligned}
\]
The following proposition establishes that $(\WI,\I)$
and $(\UU,\EI)$ are complete hereditary cotorsion pairs
in $\Mod{\Q,A}$.

\begin{prp}\label{prp:two-cotorsion}
The pairs $(\WI,\I)$ and $(\UU,\EI)$ are complete hereditary
cotorsion pairs in $\Mod{\Q,A}$, generated respectively by
the sets
\[
\begin{aligned}
 \mathcal S_0
 &=\{F_q(A/J)\mid q\in\Q,\ J\text{is a left ideal of }A\}\quad \text{and}\quad
 \mathcal S_1
 &=\mathcal S_0\cup\{S_q(A)\mid q\in\Q\}.
\end{aligned}
\]
Moreover, $\WI\subseteq\EE$ and $\WI\cap\I={}_{\Q,A}\Inj.$
\end{prp}

\begin{proof}
Let $X\in\Mod{\Q,A}$. For every $q\in\Q$ and every left
ideal $J$ of $A$, by Lemma~\ref{lem:ext-adjunction} one has
\[
 \Ext^1_{\Q,A}(F_q(A/J),X)
 \cong\Ext^1_A(A/J,X(q)).
\]
By Baer's criterion, these groups vanish for all $q$ and $J$
if and only if each $X(q)$ is injective. Thus
$\mathcal S_0^{\perp_1}=\I$.
Furthermore, the characterization of exact objects established in the proof
of \cite[Theorem~D]{HJ24} gives
\[
 \{S_q(A)\mid q\in\Q\}^{\perp_1}=\EE.
\]
It follows that
\[
 \mathcal S_1^{\perp_1}
 =\mathcal S_0^{\perp_1}
   \cap\{S_q(A)\mid q\in\Q\}^{\perp_1}
 =\I\cap\EE=\EI.
\]
Since $\Mod{\Q,A}$ is a Grothendieck category with enough
projectives, it follows from \cite[Corollary~2.15(3)]{SS11} that
$(\WI,\I)$ and $(\UU,\EI)$ are complete cotorsion pairs
generated by $\mathcal S_0$ and $\mathcal S_1$, respectively.

By Proposition~\ref{prp:I-Frobenius}, the class $\I$ is
closed under extensions and contains ${}_{\Q,A}\Inj$.
If
\[
 0\longrightarrow X\longrightarrow Y
 \longrightarrow Z\longrightarrow0
\]
is exact with $X,Y\in\I$, then it is objectwise split.
Hence each $Z(q)$ is a direct summand of $Y(q)$ and is
therefore injective. Thus $\I$ is also closed under
cokernels of monomorphisms.
Since $\EE$ is thick and contains ${}_{\Q,A}\Inj$,
the same closure properties hold for $\EI=\EE\cap\I$.
It follows that the two cotorsion pairs are hereditary.
The inclusion $\EE^{\perp_1}\subseteq\I$ and the cotorsion
pair $(\EE,\EE^{\perp_1})$ yield
\[
 \WI={}^{\perp_1}\I
 \subseteq{}^{\perp_1}(\EE^{\perp_1})
 =\EE.
\]
 Since $\Ext^1_{\Q,A}(X,I)\cong\Ext^1_{\I}(X,I)$
for $X,I\in\I$ by Lemma~\ref{lem:Ext-comparison},
 the objects of $\WI\cap\I$ are precisely the
projective objects of the exact category $\I$.
By Proposition~\ref{prp:I-Frobenius}, we obtain that $\WI\cap\I=\Prj(\I)={}_{\Q,A}\Inj.$
\end{proof}

Define
\[
 \FF=\UU\cap\I
 =\{F\in\I\mid\Ext^1_{\I}(F,Z)=0\text{ for every }Z\in\EI\}.
\]
The equality follows from Lemma~\ref{lem:Ext-comparison}.

Following \cite[Section~3.1]{SS11}, a cotorsion pair
$(\mathcal X,\mathcal Y)$ in a Frobenius exact category
$\mathcal C$ is called \emph{hereditary} if
$\Omega\mathcal X\subseteq\mathcal X$, and
\emph{cohereditary} if
$\Omega^{-1}\mathcal X\subseteq\mathcal X$.
Here $\Omega X$ and $\Omega^{-1}X$ are defined by conflations
\[
 0\to\Omega X\to P\to X\to0
 \qquad\text{and}\qquad
 0\to X\to J\to\Omega^{-1}X\to0,
\]
with $P$ and $J$ projective-injective in $\mathcal C$.
These conditions are independent of the chosen conflations.

We next show that $(\UU,\EI)$ restricts to a complete
hereditary and cohereditary cotorsion pair $(\FF,\EI)$
in $\I$. This allows us to apply
\cite[Proposition~3.9]{SS11} to obtain a Bousfield localizing
pair in $\underline{\I}\simeq\KI$.

\begin{prp}\label{prp:restricted-cotorsion}
The cotorsion pair $(\UU,\EI)$ in $\Mod{\Q,A}$ restricts
to a complete hereditary and cohereditary cotorsion pair $(\FF,\EI)$ in the Frobenius
exact category $\I$.
\end{prp}

\begin{proof}
Let $X\in\I$. Since $(\UU,\EI)$ is a complete cotorsion pair by Proposition \ref{prp:two-cotorsion},  we have two exact
sequences in $\Mod{\Q,A}$
\begin{equation*}
 0\to Z_X\to F_X\to X\to0
 \qquad\text{and}\qquad
 0\to X\to Z^X\to F^X\to0,
\end{equation*}
with $Z_X,Z^X\in\EI$ and $F_X,F^X\in\UU$.
Both sequences are  objectwise split, since $Z_X$ and $X$ are
in $\I$. Thus $F_X(q)\cong Z_X(q)\oplus X(q)$
and $Z^X(q)\cong X(q)\oplus F^X(q)$ for any $q\in \Q.$
It follows that $F_X$ and $F^X$ are in $\I$,
so both belong to $\FF$. Hence
both short exact sequences are conflations in $\I$.
Recall that
\[
 \FF=\{F\in\I\mid
 \Ext^1_{\I}(F,Z)=0\text{ for every }Z\in\EI\}.
\]
Suppose that $X\in\I$ satisfies
$\Ext^1_{\I}(F,X)=0$ for every $F\in\FF$.
The conflation $0\to X\to Z^X\to F^X\to0$
splits because $F^X\in\FF$. Thus $X$ is a direct summand
of $Z^X\in\EI$, and hence belongs to $\EI$.
Together with the two conflations constructed above, this shows that
$(\FF,\EI)$ is a complete cotorsion pair in $\I$.

It remains to prove that $(\FF,\EI)$ is hereditary and cohereditary.
For $Z\in\EI$, we have two conflations in $\I$
\[
 0\to\mathbb K(Z)\to\mathbb F(Z)\to Z\to0
 \qquad\text{and}\qquad
 0\to Z\to\mathbb G(Z)\to\mathbb C(Z)\to0
\]
with $\mathbb F(Z)$ and $\mathbb G(Z)$ projective-injective in $\I$.
By Lemma~\ref{lem:canonical-I} and \cite[Theorem 4.4]{HJ22}, we obtain that
$\mathbb K(Z)$  and $\mathbb C(Z)$ belong to $\EI$.
Moreover, these conflations give $\mathbb K(Z)\cong\Sigma^{-1}Z$ and $\mathbb C(Z)\cong\Sigma Z$ in $\KI.$
Now let $F$ be in $\FF$. By Lemma~\ref{lem:canonical-I}, both
$\mathbb K(F)$ and $\mathbb C(F)$ belong to $\I$.
It follows from Lemma~\ref{lem:Ext-comparison} that
\begin{align*}
 \Ext^1_{\I}(\mathbb K(F),Z)
 &\cong\Hom_{\KI}(\Sigma^{-1}F,\Sigma Z)\\
 &\cong\Hom_{\KI}(F,\Sigma\mathbb C(Z))\\
 &\cong\Ext^1_{\I}(F,\mathbb C(Z))=0,
\end{align*}
and
\begin{align*}
 \Ext^1_{\I}(\mathbb C(F),Z)
 &\cong\Hom_{\KI}(\Sigma F,\Sigma Z)\\
 &\cong\Hom_{\KI}(F,\Sigma\mathbb K(Z))\\
 &\cong\Ext^1_{\I}(F,\mathbb K(Z))=0.
\end{align*}
These vanishings hold for every $Z\in\EI$, so $\mathbb K(F),\mathbb C(F)\in\FF$.
Thus $\FF$ is closed under syzygies and cosyzygies,
proving that $(\FF,\EI)$ is hereditary and cohereditary.
\end{proof}

\begin{dfn}\cite{Bondarko}
Let $\DD$ be a triangulated category. A \emph{$t$-structure} in $\DD$ is a pair $(\X, \Y)$ of full subcategories satisfying the following properties:
\begin{enumerate}
\item $\mathrm{Hom}_{\DD}(X [1],Y)=0$ for all $X\in\X$ and $Y\in\Y;$
\item $\X[1]\subseteq\X$ (or $\Y[-1]\subseteq\Y$);
\item For each $M\in \DD$, there is a triangle $X[1]\to M\to Y\to X[2]$
with $X\in \X$ and $Y\in \Y.$
\end{enumerate}
A $t$-structure $(\X,\Y)$ is called a \emph{Bousfield localizing pair} if $\X$ is a triangulated subcategory of $\DD.$ Such pairs are also called semi-orthogonal decompositions.
\end{dfn}

\begin{cor}\label{cor:bousfield-pair}
The pair $({}^\perp\KAI,\KAI)$ is a Bousfield localizing
pair in $\KI$, and ${}^\perp\KAI$ consists precisely of
objects isomorphic in $\KI$ to objects of $\FF$.
\end{cor}

\begin{proof}
By Proposition~\ref{prp:restricted-cotorsion} and
\cite[Proposition~3.9]{SS11}, the images of $\FF$ and $\EI$
in $\underline{\I}$ form a Bousfield localizing pair.
Under $\underline{\I}\simeq\KI$, the image of $\EI$ is
$\KAI$, so the image of $\FF$ is its left orthogonal
${}^\perp\KAI$.
\end{proof}

\begin{thm}\label{thm:left-equivalence}
The restriction
\[
 Q_I|_{{}^\perp\KAI}:{}^\perp\KAI\longrightarrow\DQ
\]
is a triangle equivalence.
\end{thm}
\begin{proof}
Let $F$ and $F'$ be in ${}^\perp\KAI$. Then there is a triangle in $\KI$
$$F'\xrightarrow{\eta_{F'}}R_IQ_I(F')\longrightarrow C_{F'}
 \longrightarrow\Sigma F'.$$
Since $R_{I}$ is fully faithful, it follows that $Q_{I}(C_{F'})=0$, and then $C_{F'}\in \KAI$.
Applying $\Hom_{\KI}(F,-)$ to the above triangle,
we obtain natural isomorphisms
\[
\begin{aligned}
 \Hom_{\KI}(F,F')
 &\cong\Hom_{\KI}(F,R_IQ_IF')\\
 &\cong\Hom_{\DQ}(Q_IF,Q_IF').
\end{aligned}
\]
Thus $Q_I|_{{}^\perp\KAI}$ is fully faithful.

Now take $Y\in \DQ$.
Since $({}^\perp\KAI,\KAI)$ is a Bousfield localizing
pair in $\KI$ by Corollary \ref{cor:bousfield-pair}, there is a triangle in $\KI$
$$\Sigma F\to R_{I}(Y)\to Z\to \Sigma^{2}F$$ with
$Z\in\KAI$ and $F\in{}^\perp\KAI$. Applying $Q_I$ to the above triangle, we obtain that $Q_I(\Sigma F)\cong Q_IR_I(Y)\cong Y$.
Note that $\Sigma F\in {}^\perp\KAI$ as ${}^\perp\KAI$ is a triangulated subcategory of $\KI$. It follows that $Q_I|_{{}^\perp\KAI}$ is essentially surjective and hence a triangle equivalence.
\end{proof}

Choose a quasi-inverse of the equivalence in
Theorem~\ref{thm:left-equivalence}, followed by the inclusion into
$\KI$, and denote it by $L_I:\DQ\to\KI$. It is fully faithful,
its essential image is ${}^\perp\KAI$, and $Q_IL_I\cong\Id_{\DQ}$.

\begin{prp}\label{prp:left-adjoint}
There is an adjunction
\[
 \xymatrix@C=3.7pc{
 \DQ\ar@<0.7ex>[r]^{L_I}&
 \KI\ar@<0.7ex>[l]^{Q_I}},\]
where the upper arrow is the left adjoint, and $L_I$ is fully faithful.
\end{prp}
\begin{proof}
Let $X\in\KI$ and $Y\in\DQ$.  Consider the triangle
$$X\xrightarrow{\eta_X} R_IQ_I(X)\to C_{X}\to \Sigma X.$$
 As in the proof of Theorem \ref{thm:left-equivalence},  $C_{X}\in\KAI$. Since
$L_IY\in{}^\perp\KAI$, we obtain natural isomorphisms
\[
\begin{aligned}
 \Hom_{\KI}(L_IY,X)
 &\cong\Hom_{\KI}(L_IY,R_IQ_IX)\\
 &\cong\Hom_{\DQ}(Q_IL_IY,Q_IX)\\
 &\cong\Hom_{\DQ}(Y,Q_IX),
\end{aligned}
\]
The second isomorphism is due to that $(Q_{I},R_{I})$ is an adjoint pair, and the last follows by
$Q_IL_I\cong\Id_{\DQ}$. These isomorphisms are natural
in $X$ and $Y$, establishing the adjunction $(L_I,Q_I)$.
\end{proof}

The adjoint triple $(L_I,Q_I,R_I)$ yields the following
recollement. The associated unit and counit triangles
define the adjoints to the inclusion $\KAI\hookrightarrow\KI$
and induce the corresponding Verdier quotient equivalences.

\begin{thm}\label{thm:recollement}
Under Setup~\ref{setup:standing}, the inclusion $i_*:\KAI\to\KI$
and $Q_I$ extend to a recollement of triangulated categories
\[
\xymatrix@C=3.1pc{
 \KAI\ar[r]^{i_*}&
 \KI\ar@/_1.5pc/[l]_-{i^*}\ar@/^1.5pc/[l]^-{i^!}
 \ar[r]^{Q_I}&
 \DQ\ar@/_1.5pc/[l]_-{L_I}\ar@/^1.5pc/[l]^-{R_I}.}
\]
Moreover, $\KI/\KAI\simeq\DQ$ and $\KI/\Im L_I\simeq\KAI$.
\end{thm}
\begin{proof}
By Propositions \ref{prp:right-adjoint} and \ref{prp:left-adjoint}, we get that $(L_I,Q_I,R_I)$ is an adjoint triple.
So we obtain the desired recollement by (the dual of) \cite[Proposition 9.1.18]{Neeman01} and the corresponding triangle equivalences.
\end{proof}

\section{Compact objects in $\KI$ and $\KAI$}\label{sec:compact}

Throughout this section, we work under Setup~\ref{setup:standing}
and assume that $A$ is left noetherian.

\subsection{Noetherian objects and compact generators}
We recall the definitions of compact objects and compactly generated triangulated categories.

\begin{dfn}[{\cite[Definition~4.2]{HJ24}}]\label{dfn:compact-object}
Let $\mathcal A$ be an additive category with arbitrary
set-indexed coproducts. An object $C\in\mathcal A$ is
\emph{compact} if, for every set-indexed family
$\{X_i\}_{i\in I}$ in $\mathcal A$, the canonical map
\[
 \bigoplus_{i\in I}\Hom_{\mathcal A}(C,X_i)
 \longrightarrow
 \Hom_{\mathcal A}\left(C,\bigoplus_{i\in I}X_i\right)
\]
is an isomorphism. We denote by $\mathcal A^c$ the full
subcategory of compact objects.
\end{dfn}

\begin{dfn}[{\cite[Definition~4.3]{HJ24}}]\label{dfn:compact-generation}
Let $\mathcal T$ be a triangulated category with arbitrary
set-indexed coproducts.
\begin{enumerate}
\item
The category $\mathcal T$ is \emph{compactly generated}
if there is a set $\mathcal G\subseteq\mathcal T^c$ such that
\[
 \Hom_{\mathcal T}(G,\Sigma^nX)=0
 \quad\text{for all }G\in\mathcal G\text{ and }n\in\Z
 \quad\Longrightarrow\quad X\cong0.
\]
Such a set $\mathcal G$ is called a \emph{set of compact
generators}. We allow shifts in the detection condition;
equivalently, all shifts of $\mathcal G$ together form a generating
set in the unshifted convention of the cited definition.

\item
A full triangulated subcategory $\mathcal S\subseteq\mathcal T$
is \emph{thick} if it is closed under direct summands in
$\mathcal T$, and \emph{localizing} if it is closed under
arbitrary set-indexed coproducts in $\mathcal T$.
Every localizing subcategory is thick.

\item
For a collection $\mathcal G$ of objects of $\mathcal T$,
write $\Thick(\mathcal G)$ and $\Loc(\mathcal G)$ for the
smallest thick and localizing subcategories of $\mathcal T$
containing $\mathcal G$, respectively.
\end{enumerate}
\end{dfn}

We first prove that $\Mod{\Q,A}$ is locally noetherian and
that, for every noetherian object $M$ in $\Mod{\Q,A}$, the functor $\Ext^1_{\Q,A}(M,-)$
preserves direct sums.
Recall from \cite[Chapter~V, \S4]{Stenstrom75} that an object is \emph{noetherian} if its subobjects
satisfy the ascending chain condition. A Grothendieck category
is \emph{locally noetherian} if it has a set of noetherian
generators.

\begin{lem}\label{lem:locally-noetherian}
The category $\Mod{\Q,A}$ is locally noetherian. For
$M\in\Mod{\Q,A}$, the following conditions are equivalent:
\begin{eqc}
\item $M$ is noetherian;
\item $M$ is finitely generated;
\item $M$ is finitely presented;
\item $\supp M=\{q\in \Q \mid M(q)\neq 0\}$ is finite and each $M(q)$ is finitely generated over $A$.
\end{eqc}
Every such $M$ admits an exact sequence
\begin{equation*}
 0\longrightarrow K\longrightarrow P\longrightarrow M
 \longrightarrow0,
\end{equation*}
where $P$ is a finite direct sum of objects $F_{q}(A)$
and $K$ is finitely generated. Moreover,
$\Ext^1_{\Q,A}(M,-)$ preserves coproducts.
\end{lem}

\begin{proof}
For each $q\in\Q$, the object $F_{q}(A)$ has finite
support by local boundedness of $\Q$. Each value $F_{q}(A)(p)=\Q(q,p)\otimes_{\kk}A$
is a finitely generated projective $A$-module, hence
noetherian since $A$ is left noetherian. Then any ascending
chain of subobjects of $F_{q}(A)$ therefore stabilizes at each
object, and finite support ensures that the chain itself
stabilizes. Thus every $F_{q}(A)$ is noetherian. Since $\{F_{q}(A)\}_{q\in\Q}$ generate $\Mod{\Q,A}$ by \cite[Proposition 3.12]{HJ22}, it follows that $\Mod{\Q,A}$ is locally
noetherian.

Since $\Mod{\Q,A}$ is locally noetherian, its noetherian,
finitely generated, and finitely presented objects coincide;
see \cite[Chapter~V, \S4]{Stenstrom75}.
Thus \textup{(i)}--\textup{(iii)} are equivalent. Since every finitely generated object is a quotient of a finite direct sum of the objects $F_q(A)$,
it is noetherian and has finite support with finitely generated values. This proves $\textup{(ii)}\Rightarrow\textup{(iv)}$.

Suppose that \textup{(iv)} holds.
Since $M(q)$ is finitely generated for  each $q\in\supp M$,
there is an epimorphism $u_q:A^{m_q}\to M(q)$
with $m_q$ finite. Under the adjunction $(F_q,E_q)$,
these maps induce a morphism
\[
 \phi:\bigoplus_{q\in\supp M}F_{q}(A)^{\oplus m_q}
 \longrightarrow M.
\]
For each $p\in\supp M$, the component $\phi_p$ satisfies $
 \phi_p(\mathrm{id}_p\otimes v)=u_p(v)$,
where $v\in A^{m_p}$ and $\mathrm{id}_p\otimes v$ lies in the summand indexed
by $p$. Thus $\phi_p$ is surjective, and so $\phi$ is an epimorphism.
Since $\supp M$ is finite, $M$ is finitely generated, proving \textup{(ii)}.

Suppose that  $M$ is finitely generated. There is an exact sequence
\begin{equation*}
 0\longrightarrow K\longrightarrow P\longrightarrow M
 \longrightarrow0,
\end{equation*}
 with $P$ a finite direct sum of
the objects $F_q(A)$. Thus $K$ is noetherian, hence finitely
generated.
Thus we have the exact sequence
\[
 \Hom_{\Q,A}(P,-)\longrightarrow\Hom_{\Q,A}(K,-)
 \longrightarrow\Ext^1_{\Q,A}(M,-)\longrightarrow0.
\]
Since $P$ and $K$ are finitely presented,
both $\Hom_{\Q,A}(P,-)$ and $\Hom_{\Q,A}(K,-)$ preserve coproducts. Consequently, $\Ext^1_{\Q,A}(M,-)$ preserves coproducts.
\end{proof}

\begin{lem}\label{lem:coproducts-T}
The category $\KI$ has arbitrary coproducts.
\end{lem}
\begin{proof}
Since $A$ is left noetherian, direct sums of injective left $A$-modules are injective.
It follows that $\I$ is closed under coproducts.
Note that $\Mod{\Q,A}$ is a locally noetherian category by Lemma
\ref{lem:locally-noetherian}. It follows from \cite[Proposition 4.3]{Stenstrom75} that ${}_{\Q,A}\Inj$ is closed under coproducts.
By Propositions \ref{prp:I-Frobenius}, \ref{prp:stable-I} and Lemma~\ref{lem:factor-ideal}, we have
$\Prj(\I)=\Inj(\I)={}_{\Q,A}\Inj$ and
 $\underline{\I}\simeq\KI.$ Hence $\KI$ has arbitrary coproducts.
\end{proof}

By Proposition~\ref{prp:two-cotorsion}, for every $M\in\Mod{\Q,A}$ we may
choose an exact sequence in $\Mod{\Q,A}$
\begin{equation*}
 0\longrightarrow M\xrightarrow{u_M}\rho M\longrightarrow W_M
 \longrightarrow0,
\end{equation*}
where $\rho M\in\I$ and $W_M\in\WI$.
The monomorphism $u_M$ is a special $\I$-preenvelope.

\begin{prp}\label{prp:Hom-Ext-rho}
Let $M\in\Mod{\Q,A}$ and $X\in \I$. Then there is a natural isomorphism
\begin{equation*}
 \Hom_{\KI}(\Sigma^{-1}\rho M,X)
 \cong\Ext^1_{\Q,A}(M,X).
\end{equation*}
\end{prp}
\begin{proof}
Since $(\WI,\I)$ is a complete cotorsion pair in $\Mod{\Q,A}$, there is an exact sequence
\begin{equation}
\label{prp:Hom-Ext-rho.se1}
0\longrightarrow M\xrightarrow{u_M}\rho M\to W_M
 \to0
\end{equation}
with $\rho M\in\I$ and $W_M\in\WI.$
Applying $\Hom_{\Q,A}(-,X)$ to \eqref{prp:Hom-Ext-rho.se1}, we obtain the exact sequence
\[
 \Ext^1_{\Q,A}(W_M,X)\longrightarrow\Ext^1_{\Q,A}(\rho M,X)
 \longrightarrow\Ext^1_{\Q,A}(M,X)
 \longrightarrow\Ext^2_{\Q,A}(W_M,X).
\]
Heredity makes the outer terms vanish, giving
\[
 \Ext^1_{\Q,A}(\rho M,X)\cong\Ext^1_{\Q,A}(M,X).
\]
Since $\rho M,X\in\I$, it follows from Lemma~\ref{lem:Ext-comparison} that
\[
 \Hom_{\KI}(\Sigma^{-1}\rho M,X)
 \cong\Hom_{\KI}(\rho M,\Sigma X)
 \cong\Ext^1_{\Q,A}(\rho M,X),
\]
which proves the asserted natural isomorphism.
\end{proof}

\begin{thm}\label{thm:compact-I}
Assume Setup~\ref{setup:standing} and let $A$ be left noetherian.
Choose a set $\mathcal N$ of representatives of the isomorphism
classes of noetherian objects in $\Mod{\Q,A}$ and, for each
$M\in\mathcal N$, a special $\I$-preenvelope $M\to\rho M$.
Then $\mathcal G_{\mathcal N}
 =\{\Sigma^{-1}\rho M\mid M\in\mathcal N\}$
is a set of compact generators for $\KI$. In particular,
\begin{equation*}
 \KI=\Loc(\mathcal G_{\mathcal N})
 \qquad\text{and}\qquad
 (\KI)^c=\Thick(\mathcal G_{\mathcal N}).
\end{equation*}
\end{thm}

\begin{proof}
Since $(\WI,\I)$ is a complete hereditary cotorsion pair in $\Mod{\Q,A}$ by Proposition~\ref{prp:two-cotorsion},
there is an exact sequence in $\Mod{\Q,A}$
\begin{equation*}
 0\longrightarrow M\xrightarrow{u_M}\rho M\longrightarrow W_M
 \longrightarrow0,
\end{equation*}
with $\rho M\in\I$ and $W_M\in\WI$ for any $M\in \Mod{\Q,A}$.
By Proposition \ref{prp:Hom-Ext-rho}, one has a natural isomorphism
\[
 \Hom_{\KI}(\Sigma^{-1}\rho M,X)
 \cong\Ext^1_{\Q,A}(M,X).
\]
Thus, for $M\in\mathcal N$ and any family
$\{X_i\}_{i\in I}$ in $\I$, we have
\[
\begin{aligned}
 \Hom_{\KI}\left(\Sigma^{-1}\rho M,
                  \bigoplus_{i\in I} X_i\right)
 &\cong\Ext^1_{\Q,A}\left(M,\bigoplus_{i\in I} X_i\right)\\
 &\cong\bigoplus_{i\in I}\Ext^1_{\Q,A}(M,X_i)\\
 &\cong\bigoplus_{i\in I}
       \Hom_{\KI}(\Sigma^{-1}\rho M,X_i).
\end{aligned}
\]
Hence every object of $\mathcal G_{\mathcal N}$ is compact.

Suppose that $X\in\I$ satisfies
$\Hom_{\KI}(G,X)=0$ for every $G\in\mathcal G_{\mathcal N}$.
The preceding isomorphisms imply that
$\Ext^1_{\Q,A}(M,X)=0$ for every noetherian object
$M\in\Mod{\Q,A}$.
It follows from \cite[Lemma 2.5]{Krause98} that $X$ is
injective in $\Mod{\Q,A}$, and hence $X\simeq0$ in $\KI$.
Thus $\mathcal G_{\mathcal N}$ is a set of compact generators
for $\KI$.
The asserted descriptions of the localizing and thick closures
follow from \cite[Theorem 2.1]{Neeman96}.
\end{proof}

\subsection{Stalk modules and compact generators}

Since the pseudo-radical is nilpotent, every noetherian
$\Q$-shaped module admits a finite filtration whose factors
are finite direct sums of stalk modules $S_q(N)$, with
$N$ a finitely generated left $A$-module.
We use these filtrations to replace the compact generators
in Theorem~\ref{thm:compact-I} by
\[
 \{\Sigma^{-1}\rho S_q(N)\mid
   q\in\Q,\ N\in\mathcal M_A\},
\]
where $\mathcal M_A$ is a set of representatives of the
isomorphism classes of finitely generated left $A$-modules.
We first construct compact generators for $\KI$ and then
use the recollement to obtain compact generators for $\KAI$. The image of
$\DQ^{c}$ is treated in Subsection 4.3.

\begin{lem}\label{lem:radical-filtration}
Every noetherian object $N\in\Mod{\Q,A}$ has a finite filtration whose factors
are finite direct sums of stalks $S_q(M)$ with $M$ finitely generated
over $A$.
\end{lem}
\begin{proof}
Define the subobject $\rQ N$ of $N$ by
\[
 (\rQ N)(q)=\sum_{p\in\Q,\,\alpha\in\rQ(p,q)}\Im N(\alpha).
\]
For every morphism $\beta:q\to t$ and every
$\alpha\in\mathfrak r(p,q)$, $\beta\alpha\in\mathfrak r(p,t)$. Hence
\[
 N(\beta)\bigl(\operatorname{Im}N(\alpha)\bigr)
 =\operatorname{Im}N(\beta\alpha)
 \subseteq(\mathfrak rN)(t),
\]
and so
\[
 N(\beta)\bigl((\mathfrak rN)(q)\bigr)
 \subseteq(\mathfrak rN)(t).
\]
It follows that $\mathfrak rN$ is a subobject of $N$.
Iterating and using $\rQ^d=0$ gives the finite filtration
\[
 N\supseteq\rQ N\supseteq\cdots\supseteq\rQ^dN=0.
\]
Each quotient $N_j=\mathfrak r^jN/\mathfrak r^{j+1}N$
is noetherian. For every $\alpha\in\mathfrak r(p,q)$,
\[
 N(\alpha)\bigl((\mathfrak r^jN)(p)\bigr)
 \subseteq(\mathfrak r^{j+1}N)(q),
\]
so $N_j(\alpha)=0$. Thus $N_j$ factors through
$\Q/\mathfrak r$. By the strong retraction property,
\[
 (\Q/\mathfrak r)(p,q)\cong
 \begin{cases}
 \kk,&p=q,\\
 0,&p\ne q.
 \end{cases}
\]
Consequently, the componentwise identity maps give an
isomorphism
\[
 N_j\cong
 \bigoplus_{q\in\supp N_j}S_q(N_j(q)).
\]
By Lemma~\ref{lem:locally-noetherian}, the support of $N_j$
is finite and each $N_j(q)$ is finitely generated over $A$.
Hence every filtration factor has the required form.
\end{proof}

\begin{thm}\label{thm:stalk-generators}
Assume Setup~\ref{setup:standing} and let $A$ be left noetherian.
Choose a set $\mathcal M_A$ of representatives of the isomorphism
classes of finitely generated left $A$-modules. For each
$q\in\Q$ and $M\in\mathcal M_A$, choose a special
$\I$-preenvelope $S_q(M)\to\rho S_q(M)$.
Then
\begin{equation*}
 \mathcal G_I
 =\{\rho S_q(M)\mid q\in\Q,\ M\in\mathcal M_A\}
\end{equation*}
is a set of compact generators for $\KI$. In particular,
\[
 \KI=\Loc(\mathcal G_I)
 \qquad\text{and}\qquad
 (\KI)^c=\Thick(\mathcal G_I).
\]
\end{thm}

\begin{proof}
For each $q\in\Q$ and $M\in\mathcal M_A$, the stalk module
$S_q(M)$ has finite support and finitely generated values.
It is therefore a noetherian object of $\Mod{\Q,A}$ by
Lemma~\ref{lem:locally-noetherian}.
By Theorem~\ref{thm:compact-I}, $\Sigma^{-1}\rho S_q(M)$ is compact.
Since suspension preserves compact objects, every object
of $\mathcal G_I$ is compact.

To prove that $\mathcal G_I$ generates $\KI$, let $X\in\I$
satisfy $$\Hom_{\KI}(\rho S_q(M),\Sigma^nX)=0$$
 for all $q\in\Q,\ M\in\mathcal M_A,\ n\in\mathbb Z.$
It follows from Proposition~\ref{prp:Hom-Ext-rho} that
\[
\begin{aligned}
 \Ext^1_{\Q,A}(S_q(M),X)
 &\cong\Hom_{\KI}(\Sigma^{-1}\rho S_q(M),X)\\
 &\cong\Hom_{\KI}(\rho S_q(M),\Sigma X)=0.
\end{aligned}
\]
Now let $N$ be an arbitrary noetherian object of
$\Mod{\Q,A}$. By Lemma~\ref{lem:radical-filtration},
there is a finite filtration
\[
 N=N^0\supseteq N^1\supseteq\cdots\supseteq N^d=0
\]
whose factors are isomorphic to finite direct sums of
modules $S_q(M)$ with $q\in\Q$ and $M\in\mathcal M_A$.
Consequently,
\[
 \Ext^1_{\Q,A}(N^j/N^{j+1},X)=0
 \qquad\text{for }0\leq j<d.
\]
Applying $\Hom_{\Q,A}(-,X)$ to
\[
 0\longrightarrow N^{j+1}\longrightarrow N^j
 \longrightarrow N^j/N^{j+1}\longrightarrow0
\]
therefore yields an exact sequence
\[
 0\longrightarrow\Ext^1_{\Q,A}(N^j,X)
 \longrightarrow\Ext^1_{\Q,A}(N^{j+1},X).
\]
Starting with $N^d=0$, descending induction gives
$\Ext^1_{\Q,A}(N,X)=0$.
Since this category is locally noetherian, it follows from
\cite[Lemma~2.5]{Krause98} that $X$ is injective.
Hence $X\cong0$ in $\KI$, proving that $\mathcal G_I$
is a set of compact generators.
The descriptions of its localizing and thick closures
follow from \cite[Theorem~2.1]{Neeman96}.
\end{proof}

The left adjoint $i^{*}$ carries the compact generators of $\KI$ to compact gen-
erators of $\KAI$.
\begin{thm}\label{thm:compact-acyclic}
Assume Setup~\ref{setup:standing} and suppose that $A$ is left
noetherian. Then $\KAI$ is compactly generated. With $\mathcal G_I$
as in Theorem~\ref{thm:stalk-generators}, one has
\begin{align*}
 \KAI=\Loc_{\KAI}(i^*\mathcal G_I)\quad \text{and}\quad
 (\KAI)^c=\Thick_{\KAI}(i^*\mathcal G_I).
\end{align*}
\end{thm}
\begin{proof}
Since $Q_I$ preserves coproducts, its kernel $\KAI$ is closed
under coproducts in $\KI$. Thus coproducts in $\KAI$ are
computed in $\KI$, and the inclusion $i_*$ preserves them.
Since $\mathcal G_I$ is a compact
generating set for $\KI$ by Theorem~\ref{thm:stalk-generators}, it follows that $i^*G$ is compact in $\KAI$ for each $G\in \mathcal G_I$.
To see that $i^*\mathcal G_I$ generates $\KAI$, let $Z\in\KAI$ satisfy
$\Hom_{\KAI}(i^*G,\Sigma^nZ)=0$ for all $G\in\mathcal G_I$
and $n\in\Z$. Since $(i^{*},i_{*})$ is an adjoint pair,
\[
 \Hom_{\KI}(G,\Sigma^ni_*Z)
 \cong\Hom_{\KAI}(i^*G,\Sigma^nZ)=0.
\]
Since $\mathcal G_I$ generates $\KI$, we have $i_*Z\cong0$.
Thus $Z\cong0$ as $i_{*}$ is fully faithful. This implies that $i^*\mathcal G_I$ generates $\KAI$.

The last assertion is now an immediate consequence of \cite[Theorem 2.1]{Neeman96}.
\end{proof}

Since $(\UU,\EI)$ is a complete cotorsion pair in $\Mod{\Q,A}$ by Proposition~\ref{prp:two-cotorsion}, for any $M\in \Mod{\Q,A}$, there is an exact sequence
$$0\longrightarrow M\longrightarrow\sigma M
 \longrightarrow U_M\longrightarrow0$$
 with $\sigma M\in\EI$ and $U_M\in\UU.$
Let $\mathcal M_A$ be a set of representatives of the isomorphism
classes of finitely generated left $A$-modules. We will show that
the compact generators of $\KAI$ can be represented by special
$\EI$-preenvelopes of $S_{q}(M)$ with $M\in\mathcal M_A$.

\begin{prp}\label{prp:sigma-preenvelope}
For $M\in\Mod{\Q,A}$, choose a special $\EI$-preenvelope sequence
\[
 0\longrightarrow M\longrightarrow\sigma M
 \longrightarrow U_M\longrightarrow0
\]
with $\sigma M\in\EI$ and $ U_M\in\UU$.
For $Z\in\EI$, there is an isomorphism, natural in $Z$,
\begin{equation*}
 \Hom_{\KAI}(\Sigma^{-1}\sigma M,Z)
 \cong\Ext^1_{\Q,A}(M,Z).
\end{equation*}
Moreover,
$\sigma M\cong i^*\rho M$ in $\KAI$.
Consequently, $\{\sigma S_q(M)\mid q\in\Q,\ M\in\mathcal M_A\}$
is a set of compact generators for $\KAI$.
\end{prp}

\begin{proof}
Since $(\UU,\EI)$ is a complete cotorsion pair in $\Mod{\Q,A}$ by Proposition~\ref{prp:two-cotorsion}, there is an exact sequence
$$0\longrightarrow M\longrightarrow\sigma M
 \longrightarrow U_M\longrightarrow0$$
 with $\sigma M\in\EI$ and $U_M\in\UU.$
For $Z\in\EI$,  one has $
 \Ext^1_{\Q,A}(U_M,Z)=0$ and $\Ext^2_{\Q,A}(U_M,Z)=0$
since $(\UU,\EI)$ is hereditary.
Hence we get a natural isomorphism
\[
 \Ext^1_{\Q,A}(\sigma M,Z)
 \xrightarrow{\sim}\Ext^1_{\Q,A}(M,Z).
\]
Since $\sigma M$ and $Z$ are in $\EI\subseteq\I$ and $\KAI$ is a
triangulated subcategory of $\KI$, it follows from
Lemma~\ref{lem:Ext-comparison} that
\[
 \Hom_{\KAI}(\Sigma^{-1}\sigma M,Z)
 \cong\Hom_{\KI}(\sigma M,\Sigma Z)
 \cong\Ext^1_{\Q,A}(\sigma M,Z).
\]
Hence we obtain that $\Hom_{\KAI}(\Sigma^{-1}\sigma M,Z)
 \cong\Ext^1_{\Q,A}(M,Z).$

Let $Z\in\EI$. Since $(i^*,i_*)$ is an adjoint pair, it follows from Proposition~\ref{prp:Hom-Ext-rho} that
\[
\begin{aligned}
 \Hom_{\KAI}(\Sigma^{-1}i^*\rho M,Z)
 &\cong\Hom_{\KI}(\Sigma^{-1}\rho M,i_*Z)\\
 &\cong\Ext^1_{\Q,A}(M,Z)\\
 &\cong\Hom_{\KAI}(\Sigma^{-1}\sigma M,Z).
\end{aligned}
\]
These isomorphisms are natural in $Z\in \KAI$. So by
Yoneda's lemma we have
$\Sigma^{-1}i^*\rho M\cong\Sigma^{-1}\sigma M$,
and hence $i^*\rho M\cong\sigma M$.
By Theorem~\ref{thm:compact-acyclic}, the objects
$i^*\rho S_q(M)$ form a set of compact generators for $\KAI$.
Replacing them by the isomorphic objects $\sigma S_q(M)$
proves the final assertion.
\end{proof}

\subsection{Derived compact objects and the quotient}
To describe the compact objects of $\KAI$, we first identify the essential image of $\DQ^{c}$ under $L_{I}$. Neeman's localization theorem then gives the quotient description.
\begin{prp}\label{prp:rho-L}
For each $q\in\Q$, the special $\I$-preenvelope of $S_q(A)$
satisfies
\[
 \rho S_q(A)\cong L_IS_q(A)
 \qquad\text{in }\KI.
\]
Consequently, the essential image
$\mathcal P_{\Q}(A):=L_I((\DQ)^c)$ is given by
\begin{equation*}
 \mathcal P_{\Q}(A)
 =\Thick_{\KI}\{\rho S_q(A)\mid q\in\Q\}
 \subseteq(\KI)^c.
\end{equation*}
\end{prp}

\begin{proof}
Since $(\WI,\I)$ is a complete hereditary cotorsion pair by Proposition~\ref{prp:two-cotorsion}, there is an exact sequence in $\Mod{\Q,A}$
\begin{equation*}
 0\longrightarrow S_q(A)\xrightarrow{}\rho S_q(A)\longrightarrow W_{S_q(A)}
 \longrightarrow0,
\end{equation*}
with $\rho S_q(A)\in\I$ and $W_{S_q(A)}\in \WI$.
Since $S_q(A)\in{}^{\perp_1}\EE$ by
\cite[Proposition~3.2(a)]{HJ24} and $\EI\subseteq\EE$,
it follows from Proposition \ref{prp:Hom-Ext-rho} that
\[
 \Hom_{\KI}(\rho S_q(A),\Sigma Z)
 \cong\Ext^1_{\Q,A}(S_q(A),Z)=0
\]
for every $Z\in \EI.$
Note that $\KAI$ is closed under shifts.
Hence $\rho S_q(A)\in{}^\perp\KAI$.
By \cite[Proposition~3.2(a)]{HJ24}, we have $\EE^{\bot_{1}}\subseteq\I$. Hence $\WI\subseteq \EE$, and so $Q_I\rho S_q(A)\cong S_q(A)$ in $\DQ$.
By Theorem~\ref{thm:left-equivalence}, the restriction
\[
 Q_I|_{{}^\perp\KAI}:{}^\perp\KAI\xrightarrow{\sim}\DQ
\]
is a triangle equivalence with quasi-inverse $L_I$.
The two properties above therefore give $
 \rho S_q(A)\cong L_IQ_I\rho S_q(A)\cong L_IS_q(A)$ in $\KI.$

Since $Q_I$ has a right adjoint, it preserves coproducts.
Note that $(L_I, Q_I)$ is an adjoint pair. It follows from \cite[Lemma~3.5]{Krause05} that $L_I$ preserves compact objects.
By \cite[Theorem~D]{HJ24}, one has
\[
 (\DQ)^c=\Thick_{\DQ}\{S_q(A)\mid q\in\Q\}.
\]
Since $L_I$ is fully faithful and
$(\DQ)^c$ is idempotent complete, its essential image $L_I((\DQ)^c)$ is thick in $\KI$.
 Using the isomorphisms
$L_IS_q(A)\cong\rho S_q(A)$ proved above, we obtain
\[
\begin{aligned}
 L_I((\DQ)^c)
 &=\Thick_{\KI}\{L_IS_q(A)\mid q\in\Q\}\\
 &=\Thick_{\KI}\{\rho S_q(A)\mid q\in\Q\}.
\end{aligned}
\]
Since $L_I$ preserves compact objects, we have
$L_I((\DQ)^c)\subseteq(\KI)^c$.
\end{proof}

\begin{thm}\label{thm:compact-quotient}
Assume Setup~\ref{setup:standing} and let $A$ be left noetherian.
Let $i^*:\KI\to\KAI$ and $L_I:\DQ\to\KI$ be the functors
in the recollement of Theorem~\ref{thm:recollement}, and set
\[
 \mathcal P_{\Q}(A)
 =L_I((\DQ)^c)
 =\Thick_{\KI}\{\rho S_q(A)\mid q\in\Q\}.
\]
Then the functor $i^*$ preserves compact objects, and its restriction
$(\KI)^c\to(\KAI)^c$ has kernel $\mathcal P_{\Q}(A)$.
It induces a triangle equivalence up to direct summands
\[
 \overline{i^*}:
 (\KI)^c/\mathcal P_{\Q}(A)
 \longrightarrow(\KAI)^c,
 \qquad C\longmapsto i^*C.
\]
In particular, every object of $(\KAI)^c$ is a direct summand of $i^*C$
for some $C\in(\KI)^c$, and
$\overline{i^*}$ sends $\rho S_q(M)$ to an
object isomorphic to $\sigma S_q(M)$ for every $q\in\Q$
and every finitely generated left $A$-module $M$.
\end{thm}

\begin{proof}
 By Theorem~\ref{thm:recollement},
$\Ker i^*=\Im L_I$, and $i^*$ induces a triangle equivalence
$\KI/\Im L_I\simeq\KAI$.
Since the inclusion $i_*:\KAI\to\KI$ preserves coproducts, it follows from \cite[Lemma 3.5]{Krause05} that $i^*$ preserves compact objects.
We next show that $\Im L_I$ is generated, as a localizing
subcategory of $\KI$, by $\{L_IS_q(A)\mid q\in\Q\}$. Set
\[
 \mathcal T=\Loc_{\KI}\{L_IS_q(A)\mid q\in\Q\}.
\]
Since $L_I$ is fully faithful,
$\Im L_I$ is a triangulated subcategory of $\KI$.
Moreover, $L_I$ preserves coproducts, so $\Im L_I$ is
localizing. Note that $\Im L_I$ contains $L_IS_q(A)$ for each $q\in\Q$.
It follows that
$\mathcal T\subseteq\Im L_I$.
Conversely, consider the full subcategory
\[
 \mathcal U=\{Y\in\DQ\mid L_IY\in\mathcal T\}.
\]
Then $\mathcal U$ is a localizing subcategory of $\DQ$ containing
 $S_q(A)$ for each $q\in\Q$.
Since $\DQ=\Loc\{S_{q}(A)\mid q\in\Q\}$ by
\cite[Theorem~D]{HJ24}, it follows that $\mathcal U=\DQ$, proving
$\Im L_I\subseteq\mathcal T$. Consequently, $
 \Im L_I
 =\Loc_{\KI}\{L_IS_q(A)\mid q\in\Q\}.$

By Proposition~\ref{prp:rho-L}, each $L_IS_q(A)$ is compact
in $\KI$. We now identify the objects of $\Im L_I$ that
are compact in $\KI$. Let $X\in\Im L_I\cap(\KI)^c$.
Then $X\cong L_I(Y)$ for some $Y\in\DQ$. For every set-indexed family
$\{Y_j\}_{j\in J}$ in $\DQ$, there are natural isomorphisms
\[
\begin{aligned}
 \Hom_{\DQ}\Bigl(Y,\bigoplus_{j\in J}Y_j\Bigr)
 &\cong\Hom_{\KI}\Bigl(L_IY,
 L_I\bigl(\bigoplus_{j\in J}Y_j\bigr)\Bigr)\\
 &\cong\Hom_{\KI}\Bigl(X,\bigoplus_{j\in J}L_IY_j\Bigr)\\
 &\cong\bigoplus_{j\in J}\Hom_{\KI}(X,L_IY_j)\\
 &\cong\bigoplus_{j\in J}\Hom_{\DQ}(Y,Y_j).
\end{aligned}
\]
Hence $Y\in(\DQ)^c$, and therefore
$X\in L_I((\DQ)^c)$.
Conversely, if $Y\in(\DQ)^c$, then $L_I(Y)$ is compact in
$\KI$ by Proposition~\ref{prp:rho-L}. Thus we obtain
\[
 \Im L_I\cap(\KI)^c
 =L_I((\DQ)^c)
 =\mathcal P_{\Q}(A).
\]
The $\KI$ is compactly generated by Theorem~\ref{thm:compact-I}, and $\Im L_I$ is the localizing subcategory generated by the compact objects $L_IS_q(A)$.
Now applying the localization theorem of Neeman-
Ravenel-Thomason-Trobaugh-Yao
\cite[Theorem~2.1]{Neeman92}, there is  a fully faithful
triangle functor
\[
 (\KI)^c\big/\bigl(\Im L_I\cap(\KI)^c\bigr)
 \longrightarrow(\KI/\Im L_I)^c
\]
such that every object of $(\KI/\Im L_I)^c$ is a direct summand of
an object in its image. Using the equivalence
$\KI/\Im L_I\simeq\KAI$ induced by $i^*$ and the preceding
identification of $\Im L_I\cap(\KI)^c$, this is the asserted
functor $\overline{i^*}$.
The description of its action on $\rho S_q(M)$ follows from
Proposition~\ref{prp:sigma-preenvelope}.
\end{proof}

Thus $(\KAI)^c$ is triangle equivalent to the idempotent
completion of $(\KI)^c/\mathcal P_{\Q}(A)$.
This description does not require an equivalence between
$(\KI)^c$ and the bounded derived category of noetherian
$\Q$-shaped modules.

\begin{prp}[{\cite[Corollary 6.13]{SV26}}]\label{prp:finite-gldim}
Under Setup~\ref{setup:standing}, if $A$ has finite left global
dimension, then
\[
 \KAI=0\qquad\text{and}\qquad Q_I:\KI\xrightarrow{\sim}\DQ.
\]
This assertion does not require $A$ to be noetherian.
\end{prp}
\begin{proof}
By \cite[Theorem~E]{HJ24}, finite left global dimension implies
$\EE^{\perp_1}=\I$. Consequently
\[
 \EI=\EE\cap\I=\EE\cap\EE^{\perp_1}={}_{\Q,A}\Inj.
\]
Its stable category is zero. Equivalently, the kernel of $Q_I$ is
zero. The quotient equivalence $\KI/\KAI\simeq\DQ$ in Theorem~\ref{thm:recollement} therefore
makes $Q_I$ a triangle equivalence. This also recovers
\cite[Corollary~6.13]{SV26}.
\end{proof}

\begin{exa}\label{exa:ordinary}
\begin{enumerate}
\item[(1)]  Let $\Q$ be the $\kk$-linear path category of the quiver
\[
 \cdots\longrightarrow q+1\longrightarrow q
 \longrightarrow q-1\longrightarrow\cdots
\]
modulo all paths of length two. Then $\Mod{\Q,A}$ is
the category of chain complexes of left $A$-modules,
and there are natural identifications
\[
 \KI\simeq\mathcal K({}_A\Inj),\quad
 \KAI\simeq\mathcal K^{\mathrm{ac}}({}_A\Inj),
 \quad\text{and}\quad
 \DQ\simeq\mathcal D(A).
\]
Under these identifications, $Q_I$ is the canonical
localization functor. Hence Theorem~\ref{thm:recollement}
gives the recollement
\[
\xymatrix@C=3pc{
 \mathcal K^{\mathrm{ac}}({}_A\Inj)\ar[r]^{i_*}&
 \mathcal K({}_A\Inj)
 \ar@/_1.5pc/[l]_-{i^*}\ar@/^1.5pc/[l]^-{i^!}
 \ar[r]^{Q_I}&
 \mathcal D(A)
 \ar@/_1.5pc/[l]_-{L_I}\ar@/^1.5pc/[l]^-{R_I}}.
\]
Assume now that $A$ is left noetherian.
The preceding recollement recovers
\cite[Corollary~4.3]{Krause05}.
Theorems~\ref{thm:compact-I} and~\ref{thm:compact-acyclic}
show that $\mathcal K({}_A\Inj)$ and
$\mathcal K^{\mathrm{ac}}({}_A\Inj)$ are compactly generated.
By \cite[Proposition~2.3 and Lemma~5.2]{Krause05},
the restriction of $Q_I$ gives a triangle equivalence $
 \mathcal K({}_A\Inj)^c
 \xrightarrow{\ \sim\ }\mathcal D^b({}_A\modu)$
that identifies $L_I(\mathcal D(A)^c)$ with the full
subcategory $\mathcal K^b({}_A\prj)$ of perfect complexes.
Theorem~\ref{thm:compact-quotient} therefore induces a
fully faithful triangle functor
\[
 \mathcal D^b({}_A\modu)/\mathcal K^b({}_A\prj)
 \longrightarrow
 \mathcal K^{\mathrm{ac}}({}_A\Inj)^c,
\]
and every compact object of
$\mathcal K^{\mathrm{ac}}({}_A\Inj)$ is a direct summand
of an object in its essential image.
Thus $\mathcal K^{\mathrm{ac}}({}_A\Inj)^c$ is triangle
equivalent to the idempotent completion of the singularity
category of $A$, recovering \cite[Corollary~5.4]{Krause05}.

\item[(2)] For $m\geq1$, let $\Q_m$ be the $\kk$-linear path category of the quiver

\[
\begin{tikzpicture}[
  every node/.style={inner sep=2.5pt},
  every path/.style={->, bend right=18}
]
  \node (q0)    at (0:2.3)   {$0$};
  \node (qm1)   at (60:2.3)  {$m-1$};
  \node (qm2)   at (120:2.3) {$m-2$};
  \node (qdots) at (180:2.3) {$\vdots$};
  \node (q2)    at (240:2.3) {$2$};
  \node (q1)    at (300:2.3) {$1$};

  \draw (q0)    to (qm1);
  \draw (qm1)   to (qm2);
  \draw (qm2)   to (qdots);
  \draw (qdots) to (q2);
  \draw (q2)    to (q1);
  \draw (q1)    to (q0);
\end{tikzpicture}
\]
modulo all paths of length two. Then $\Mod{\Q_m,A}$ is
the category of $m$-periodic complexes of left $A$-modules.
By \cite[Example 1.5]{HJ242} and Theorem~\ref{thm:recollement}, we get the following recollement
\[
\xymatrix@C=3pc{
 \mathcal K_m^{\mathrm{ac}}({}_A\Inj)\ar[r]^{i_*}&
 \mathcal K_m({}_A\Inj)
 \ar@/_1.5pc/[l]_-{i^*}\ar@/^1.5pc/[l]^-{i^!}
 \ar[r]^{Q_I}&
 \mathcal D_m(A)
 \ar@/_1.5pc/[l]_-{L_I}\ar@/^1.5pc/[l]^-{R_I}},
\]
with $Q_I$ the canonical periodic localization.
When $A$ is left noetherian, we recover Liu's recollement
\cite[Theorem~3.1]{Liu23}.
By Theorems~\ref{thm:compact-I} and~\ref{thm:compact-acyclic},
its middle and left terms are compactly generated, while
the stalk modules $\{S_q(A)\mid q\in\Q\}$ form a set of
compact generators for its right term.

\item[(3)] Let $\Q$ be the $\kk$-linear path category of the quiver
\[
 \cdots\longrightarrow q+1\longrightarrow q
 \longrightarrow q-1\longrightarrow\cdots
\]
modulo all paths of length $N$.
The category $\Mod{\Q,A}$ identifies with the category
of $N$-complexes of left $A$-modules, and the objects of
$\EE$ are precisely the $N$-acyclic complexes.
By \cite[Example 1.5]{HJ242} and Theorem~\ref{thm:recollement}, we get the following recollement
\[
\xymatrix@C=2.8pc{
 \mathcal K_N^{\mathrm{ac}}({}_A\Inj)
 \ar[r]^{i_*}
 &
 \mathcal K_N({}_A\Inj)
 \ar@/_1.5pc/[l]_-{i^*}
 \ar@/^1.5pc/[l]^-{i^!}
 \ar[r]^{Q_I}
 &
 \mathcal D_N(A)
 \ar@/_1.5pc/[l]_-{L_I}
 \ar@/^1.5pc/[l]^-{R_I}.}
\]
Assume that $A$ is left noetherian, and write $S_i(M)$
for the stalk $N$-complex with $M$ in degree $i$.
By Theorem~\ref{thm:B}, the sets
\[
\begin{aligned}
 \mathcal G_I=\{\rho S_i(M)\mid i\in\mathbb Z,\ M\in\mathcal M_A\}\quad \text{and} \quad \mathcal G_{\mathrm{ac}}=\{\sigma S_i(M)\mid i\in\mathbb Z,\ M\in\mathcal M_A\}
\end{aligned}
\]
form sets of compact generators for
$\mathcal K_N({}_A\Inj)$ and
$\mathcal K_N^{\mathrm{ac}}({}_A\Inj)$, respectively.
For $\mathcal K_N({}_A\Inj)$, this recovers the existence
of compact generators established in
\cite[Theorem~3.23]{IKM17}.

To describe the compact objects of
$\mathcal K_N^{\mathrm{ac}}({}_A\Inj)$, recall that
$Q_I$ restricts to a triangle equivalence $
 \mathcal K_N({}_A\Inj)^c
 \xrightarrow{\ \sim\ }\mathcal D_N^b({}_A\modu)$
by \cite[Theorem~3.23]{IKM17}.
Furthermore,  by \cite[Theorem~4.2]{IKM17} and \cite[Proposition~4.15(i)]{IKM17}, we
identify $\mathcal D_N(A)^c$ with
$\mathcal K_N^b({}_A\prj)$.
The isomorphism $Q_IL_I\simeq\Id_{\mathcal D_N(A)}$
therefore shows that the preceding equivalence identifies
$L_I(\mathcal D_N(A)^c)$ with
$\mathcal K_N^b({}_A\prj)$.
Applying Theorem~\ref{thm:compact-quotient}, we obtain
a fully faithful triangle functor
\[
 \mathcal D_N^b({}_A\modu)/\mathcal K_N^b({}_A\prj)
 \longrightarrow
 \mathcal K_N^{\mathrm{ac}}({}_A\Inj)^c.
\]
Every compact object of
$\mathcal K_N^{\mathrm{ac}}({}_A\Inj)$ is a direct
summand of an object in its essential image.
Thus $\mathcal K_N^{\mathrm{ac}}({}_A\Inj)^c$ is
triangle equivalent to the idempotent completion of
the $N$-singularity category
$\mathcal D_N^b({}_A\modu)/\mathcal K_N^b({}_A\prj)$.
\end{enumerate}
\end{exa}

\end{document}